\documentclass[a4paper]{article}
\usepackage[english]{babel}
\usepackage[T1]{fontenc}

\usepackage[affil-it]{authblk}
\usepackage{amsmath, amsthm}
\usepackage{graphicx}
\usepackage{pgfplots}
\usepackage{mathptmx}
\usepackage{dsfont}
\usepackage[colorinlistoftodos]{todonotes}
\usepackage{braket,mleftright}
\usepackage{ upgreek }
\usepackage{tikz}
\usepackage{mathptmx}
\usetikzlibrary{arrows}
\usetikzlibrary{decorations.markings}
\usepackage{amssymb}
\usepackage{cite}

\usepackage[utf8]{inputenc}

\newcommand{\Rn}{\mathds{R}^{n}}
\newcommand{\R}{\mathds{R}}
\newcommand{\vect}[1]{\mathbf{#1}}

\newcommand{\norm}[1]{\left\lVert#1\right\rVert}
\usepackage{xcolor}

\newtheorem{definition}{Definition}
\newtheorem{lemma}{Lemma}
\newtheorem{corollary}{Corollary}
\newtheorem{theorem}{Theorem}
\newtheorem{proposition}{Proposition}
\newtheorem*{lem}{Lemma}
\theoremstyle{remark}
\newtheorem{remark}{Remark}

\makeatletter
\def\@maketitle{%
  \newpage
  \null
  \vskip 2em%
  \begin{center}%
  \let \footnote \thanks
    {\Large\bfseries \@title \par}%
    \vskip 1.5em%
    {\normalsize
      \lineskip .5em%
      \begin{tabular}[t]{c}%
        \@author
      \end{tabular}\par}%
    \vskip 1em%
    {\normalsize \@date}%
  \end{center}%
  \par
  \vskip 1.5em}
\makeatother

\begin{document}
\title{Deterministic versus stochastic consensus dynamics on graphs}

\author{Dylan Weber \thanks{Electronic address: \texttt{djweber3@asu.edu}}}
\affil{School of Mathematical and Statistical Sciences, Arizona State University, Tempe}

\author{Ryan Theisen \thanks{Electronic address: \texttt{theisen@berkeley.edu}}}
\affil{Department of Statistics, University of California, Berkeley}

\author{Sebastien Motsch \thanks{Electronic address: \texttt{smotsch@asu.edu}}}
\affil{School of Mathematical and Statistical Sciences, Arizona State University, Tempe}

\date{Dated: \today}
\title{Deterministic versus stochastic consensus dynamics on graphs}

\maketitle

\begin{abstract}
We study two agent based models of opinion formation - one stochastic in nature and one deterministic.  Both models are defined in terms of an underlying graph; we study how the structure of the graph affects the long time behavior of the models in all possible cases of graph topology.  We are especially interested in the emergence of a \textit{consensus} among the agents and provide a condition on the graph that is necessary and sufficient for convergence to a consensus in both models.  This investigation reveals several contrasts between the models - notably the convergence rates - which are explored through analytical arguments and several numerical experiments.
\end{abstract}

\tableofcontents

\section{Introduction}


Mathematical models involving dynamics defined on network structures have long been objects of theoretical interest.  The rise of social networks as the ubiquitous forum for the exchange of information in our society and the meta data associated with these networks clearly demonstrates the need to 
analyze dynamics on networks. Such models are often posed in an agent based framework where the potential for agents to interact is encoded in a \textit{graph}; models of opinion formation are especially amenable to such a framework and a considerable amount of research has been done in this context \cite{olfati-saber_consensus_2007, xia_opinion_2011, lorenz_continuous_2007, castellano_statistical_2009, hegselmann_opinion_2002}.  Many of the dynamics studied are deterministic in nature \cite{spanos_dynamic_2005, blondel_krauses_2009, ben-naim_unity_2003, hegselmann_opinion_2002}, however  the random nature of information exchange among humans strongly motivates the study of similar models in a stochastic setting \cite{deffuant_mixing_2000, aldous_interacting_2013, lanchier_stochastic_2013}.  In this paper we investigate two models of opinion formation, one stochastic and one deterministic with special attention paid to the emergence of \textit{consensus} - when the opinions of all agents agree.

A hallmark of the agent-based approach is the investigation of how locally defined interactions affect global behavior observed among the entire population of agents.  In the context of opinion formation, local interactions are determined by the underlying network structure; therefore the investigation of the long time behavior of the dynamics translates to understanding how the topology of the underlying graph affects the distribution of opinions among the agents.  Of particular interest is how the interplay between the local ``rules of engagement'' and the structure of the graph affect the emergence of a consensus among the agents \cite{motsch_heterophilious_2014, saber_consensus_2003}. Deterministic models often carry the generic assumption that the opinion of a given agent is continuously influenced by those to whom it is connected in the graph.  These connections can be considered to be static or changing often depending on the state of agents. The manner in which influence is exerted is usually globally defined - most models carry the assumption of local consensus i.e. if agents interact they agree in some sense.  As the analysis of these dynamics can often be quite delicate, much of the analytical work done in the past uses assumptions on the graph such as symmetry of connections or connectedness of the graph \cite{saber_consensus_2003, spanos_dynamic_2005, yu_second-order_2010, olfati-saber_consensus_2007}.  However, in many cases the dynamics are too complex for analytical proof and main characteristics of the dynamics are identified through numerical simulation \cite{carro_role_2012, ben-naim_unity_2003, motsch_heterophilious_2014, deffuant_comparing_2006}.  A main takeaway in deterministic models is that a sufficient condition for consensus is the persistence of a suitable amount of connectivity in the underlying graph throughout the evolution of the dynamics \cite{yu_second-order_2010, ren_consensus_2005, jabin_clustering_2014}.

To incorporate stochastic behavior into opinion dynamics, several strategies are available. For instance, given a deterministic model, opinions can be perturbed with white noise \cite{pineda_noisy_2013, carro_role_2012} to model uncertainty in the agent.  Additionally the edges of the interaction graph might change randomly (possibly in a state dependent manner) to model uncertainty in agent interaction \cite{von_brecht_swarming_2012, deffuant_mixing_2000,  castellano_incomplete_2003, ben-naim_unity_2003}.  The case where the interaction graph remains fixed, and interactions among neighbors are randomized is especially well studied as the classical voter model falls into this class.  In this case assumptions on the structure of the graph are usually made - often that it is a regular lattice. The stochastic nature of these models causes their analysis to require a different set of mathematical tools, most frequently from the theory of Markov chains and martingales. Similar to the deterministic case, a sufficient degree of interaction among agents is often necessary for a consensus to emerge \cite{lanchier_stochastic_2017, lanchier_stochastic_2013, aldous_reversible_2002}.






In both cases we are interested in how the structure of the graph affects the long-time behavior of the model.  We are particularly interested in conditions on the graph that cause the opinion of every agent to converge to a \textit{consensus} - i.e. every agent converges on the same opinion in the long-time limit. Moreover we analyze if the addition of stochasticity helps or hinders the emergence of a consensus; if consensus occurs, we are interested in how the connectivity of the graph and stochasticity affect the speed of convergence. We examine all possible graph topologies: i) undirected and connected (referred to as \textit{simple} graph), ii) directed and strongly connected, and iii) directed and weakly connected or disconnected.

In the deterministic case we find that the model converges regardless of graph structure.  In the case that the underlying graph is simple, we find that a consensus is reached that is the \textit{initial average} of opinions. Similarly, if the graph is strongly connected, but not necessarily undirected, we again find that a consensus emerges, however it is not in general the initial average. The most challenging case occurs when the underlying network is only weakly connected. In general, a consensus does not form. However, we are able to find necessary and sufficient conditions on the structure of the graph for convergence to a consensus regardless of the initial condition: the graph needs to have only one \textit{isolated strongly connected component}.


In the stochastic case we similarly find consensus in the case of a strongly connected network using markedly different techniques in the arguments.  However in contrast to the deterministic model, the stochastic model does not converge in general in the case of a weakly connected graph. Despite this, we find that the condition that the graph have exactly one isolated strongly connected component prevails and is equivalent to convergence to a consensus in the stochastic case as well.  We explore this link between the models further as we find that the deterministic dynamics can be recovered in expectation from the stochastic dynamics.

We discuss speed of convergence for both the deterministic and stochastic models and quantify it in terms of the \textit{algebraic connectivity} of the graph, also known as the Fielder number\cite{fiedler_algebraic_1973, kolokolnikov_maximizing_2015, kolokolnikov_algebraic_2014}, which also happens to be the spectral gap of the operator associated with consensus dynamics.  The algebraic connectivity measures in a sense, how \textit{well connected} the graph is and we find generally that better connection in the graph results in faster convergence of the models.  However, we find a contrast between the models as the convergence rate depends on the number of agents in the stochastic case whereas the rate in the deterministic case depends only on the connectivity of the graph. In short, the spectral gap reduces drastically in the stochastic model. 

This study could be extended through the inclusion of more subtle assumptions about agent-agent interactions via non-linear dynamics.  Adding nonlinearities significantly complicates the investigation of consensus as it becomes more difficult to find sufficient conditions for the graph to remain suitably connected \cite{motsch_heterophilious_2014}. We are also interested in examining the dynamics from a control perspective - we include a brief comment discussing how given control over the network structure, the deterministic dynamics we study can be used as an algorithm to induce consensus to any opinion in the convex hull of initial agent opinions \cite{aydogdu_interaction_2017}. Finally, we aim to study limiting behavior of the model as the number of agents approaches infinity to develop a global description of the dynamics through a partial differential equation derived possibly through a kinetic description of the model \cite{dornic_critical_2001, herty_large_2014}.

\section{Opinion dynamics: a deterministic model}

\subsection{Model definition and preliminary properties}

We first study the following deterministic consensus model.
\begin{definition}
Given a collection of N agents, let $s_{i}\in \mathbb{R}$ represent the opinion of the $i$th agent.  The \textbf{consensus model} (CM) is defined by the dynamics:
\begin{equation}\label{model_agent}
  s_{i}'=\sum\limits_{j\neq i}a_{ij}(s_{j}-s_{i})\quad\text{where}\quad a_{ij}\geq0
\end{equation}
In vector notation the dynamics can be written as:
\begin{equation}\label{model_vector}
  \dot{\vect{s}}=-L\vect{s},
\end{equation}
where
  \begin{equation}
    \label{additive_condition}
    L = \left[
      \begin{array}{cccc}
        \sigma_1 & & & \\
                 &\ddots & -a_{ji} & \\
                 & -a_{ij} & \ddots & \\
                 &  & & \sigma_N
      \end{array}
    \right]
  \end{equation}
  and
  \begin{equation}  \label{eq:sigma_deter}
    \sigma_i = \sum_{j=1,j\neq i}^N a_{ij}.
  \end{equation}
\end{definition}
For the sake of clarity of presentation we restrict opinions to being one dimensional. We note that the consensus model can be easily extended for opinions $s_i$ in $\mathbb{R}^n$ and that all results discussed will still hold with analogous proofs.



 If we interpret $a_{ij}$ as measuring the amount of influence agent $j$ exerts on agent $i$, (in particular, in this interpretation if $a_{ij}=0$ then agent $j$ does not influence agent $i$) then we can interpret the matrix $L$ as encoding the structure of a network on which the agents are interacting.  In this interpretation, each vertex of the graph $G$ represents an agent and an edge from vertex $i$ to vertex $j$ represents that agent $i$ is exerting influence on agent $j$ directly, i.e. that $a_{ji}>0$.  When one agent influences another directly, the influenced agent adjusts its opinion towards the influencer's at a rate proportional to how much influence is being exerted.

 \begin{remark}
We note that $L$ is the transpose of the \textit{Laplacian matrix} of the directed graph $G$ (with the convention that the diagonal entries of $L$ measure the \textit{indegree} of the corresponding vertex in $G$.) 
\end{remark}

In this paper we will study how different conditions on the matrix $L$ affect the behavior of the model.  We will mostly be concerned with whether a \textit{consensus} is formed among the agents, i.e. if the opinions of every agent converge to the same value.

\begin{definition}
  We say that the consensus model converges to a \textbf{consensus} if there exists $\alpha\in\R$ such that:
  \begin{equation}\label{consensus}
    \lim_{t\rightarrow +\infty}\vect{s}(t)=\alpha {\bf 1},
  \end{equation}
  where ${\bf 1}=(1,1,..,1)^{T}$.
\end{definition}

We will also be interested in convergence to a consensus that is \textit{unconditional}.

\begin{definition}
The consensus model \textbf{converges to a consensus unconditionally} if a consensus is reached for any choice of the initial state $\vect{s}(0)$.
\end{definition}

Since $L$ encodes the structure of a network, the conditions we place on $L$ can be viewed as conditions on the structure of $G$, the network on which the agents are interacting.  We will show that in the case where no extra assumptions are made about $L$ that the model always converges, although not necessarily to a consensus. We find that in the case where $L$ is irreducible that the model converges unconditionally to a consensus; if we add the assumption that $L$ is symmetric we find that the opinion the agents convene on is the mean of their \textit{initial} opinions.  Assuming that $L$ is irreducible is equivalent to assuming that $G$ is a directed graph which is \textit{strongly connected}.  Adding the assumption that $L$ is symmetric is equivalent to assuming that $G$ is a connected undirected graph (often called a \textit{simple graph}) where every interaction between agents is mutual.  However, we find that these assumptions - while sufficient for unconditional convergence to a consensus - aren't necessary and can be weakened. We finally present conditions on $L$ that are equivalent to unconditional convergence to a consensus and are weaker then irreducibility.  In terms of $G$ these conditions are equivalent to demanding that $G$ be  \textit{weakly connected} with only one isolated strongly connected component.  In Fig.~\ref{fig:four_cases}, we give an illustration of these cases of study for $L$ and two examples of the time evolution of opinions are plotted in Fig.~\ref{fig:evolution_2_cases}. Similar work including nonlinearities in the case of undirected connected graphs can be found in \cite{saber_consensus_2003,spanos_dynamic_2005}.

    \begin{figure}[p]
      \centering
      \includegraphics[scale = .5]{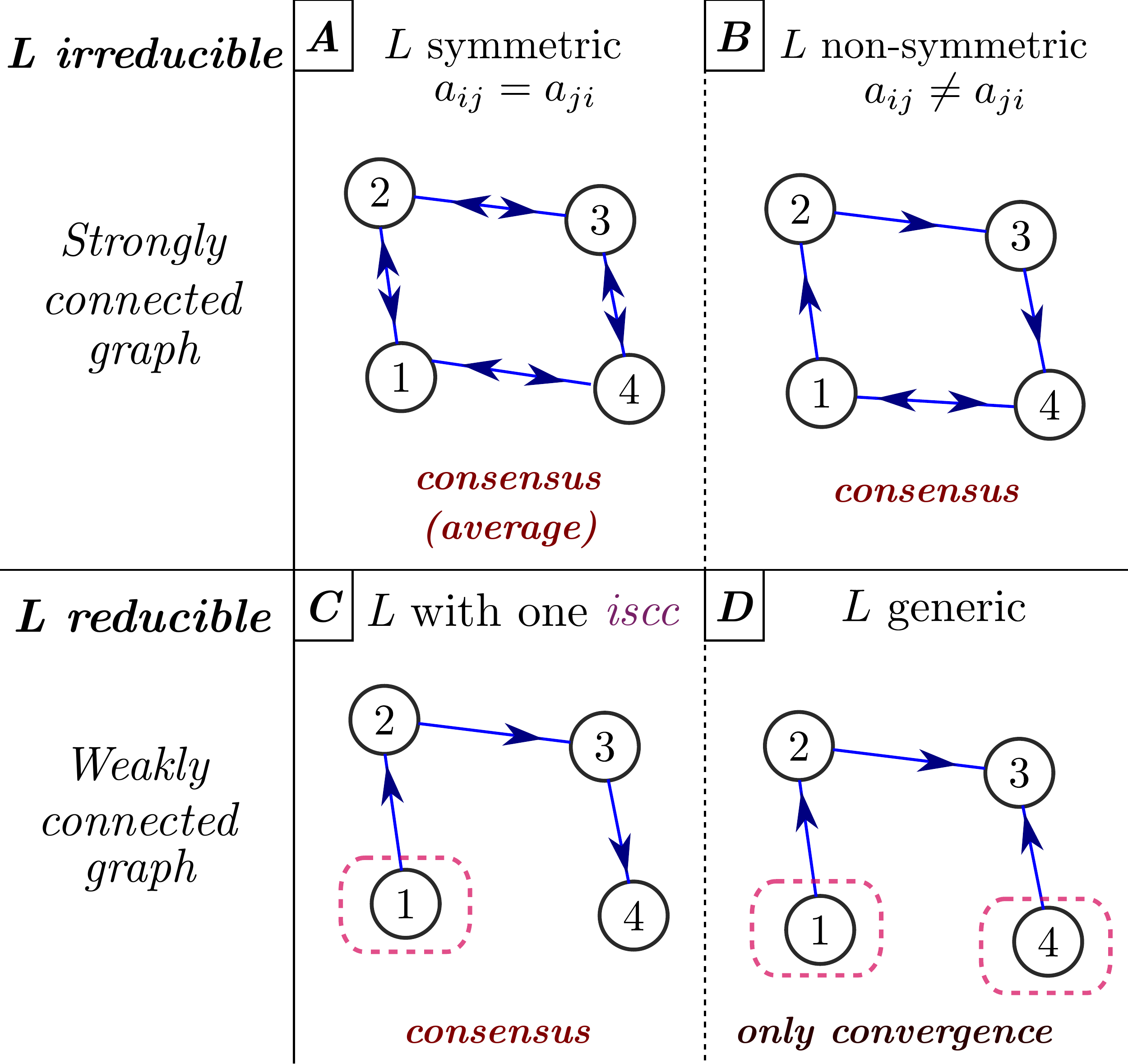}
      \caption{The four cases (A, B, C, D)  for the structure of the network: from the strongest assumption on $L$ (symmetric irreducible) to a generic matrix ({\it weakly} connected or disconnected). A sufficient condition to reach consensus is to have $L$ with only one {\it isolated strongly connected component}.}
      \label{fig:four_cases}
    \end{figure}

    \begin{figure}[p]
      \centering
      \includegraphics[width=.47\textwidth]{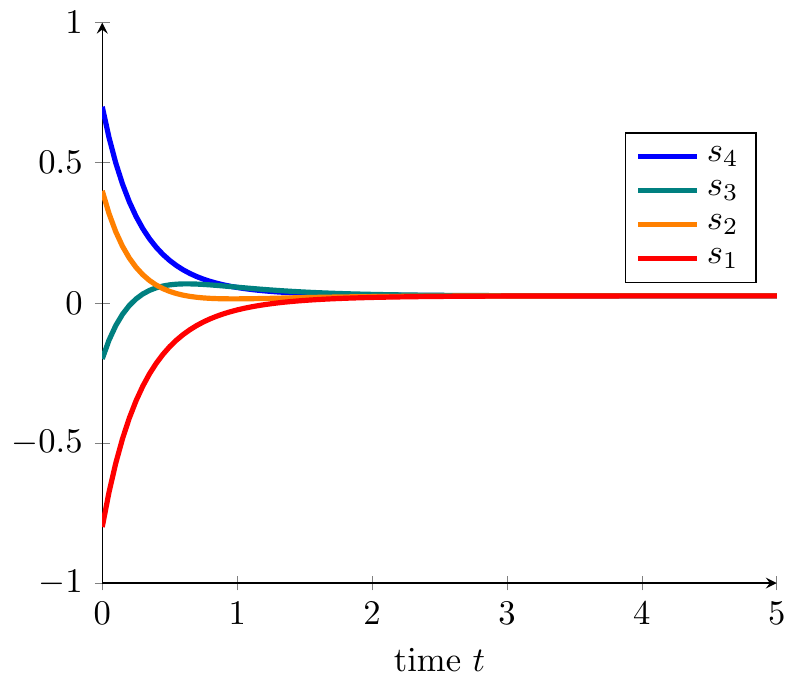} \quad
      \includegraphics[width=.47\textwidth]{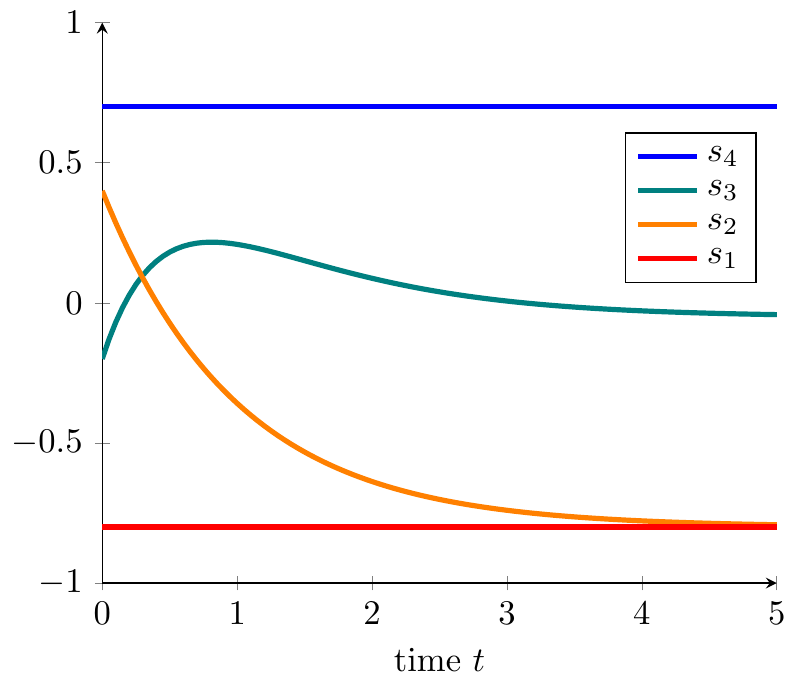}
      \caption{Evolution of the opinions $s_{i}(t)$ using an irreducible symmetric $L$ (case A figure \ref{fig:four_cases}) and a reducible $L$  (case D figure \ref{fig:four_cases}) from Fig.~\ref{fig:four_cases}. Notice that agents $\set{1}$ and $\set{4}$ do not directly or indirectly influence each other.}
      \label{fig:evolution_2_cases}
    \end{figure}

Before investigating the cases discussed above we present some general properties of $L$.  $L$ has a special structure without making any assumptions on its coefficients.

\begin{proposition}
  \label{ppo:gershgorin}
  The matrix $L$ is a diagonal dominant matrix, its eigenvalues satisfy $\text{Re}(\lambda_i)\geq 0$ and no eigenvalue $\lambda_i$ is purely imaginary except zero.
\end{proposition}
\begin{proof} The diagonal entries of $L$ satisfy $\sigma_i=\sum_{j,j\neq i} a_{ij}$ where $a_{ij}>0$, thus summing over each row will give zero and we deduce that $L$ is diagonal dominant.  By the Gershgorin disc theorem the eigenvalues $\lambda_i$ are contained in the closed discs $B(\sigma_i,\sigma_i)$ (see Fig.~\ref{fig:gershgorin}). Thus, the eigenvalues $\lambda_i$ are either $0$ or have a real part strictly positive. Moreover, $0$ is always an eigenvalue of $L$ as the constant vector ${\bf 1}=(1,\dots,1)^T$ is always an eigenvector of $L$ associated with the eigenvalue $\lambda=0$. 
\end{proof}

\begin{figure}[ht]
  \centering
  \includegraphics[width=.47\textwidth]{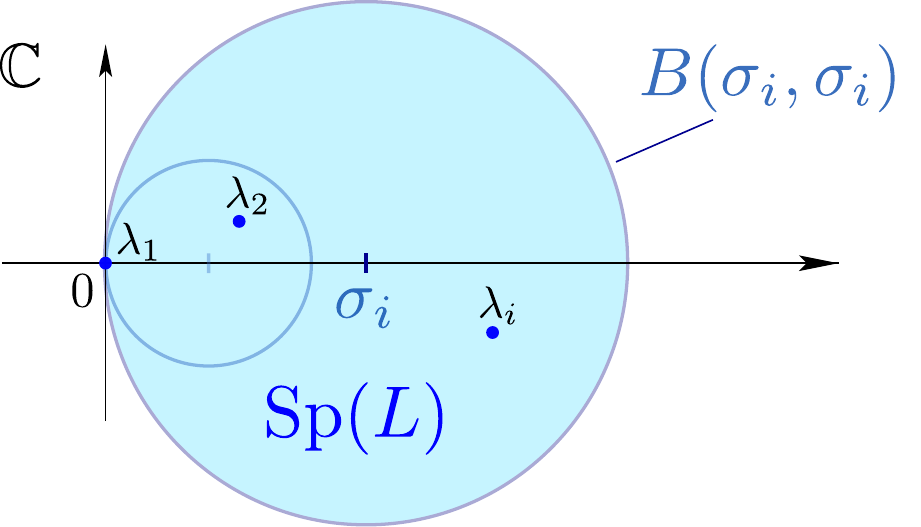}
  \caption{The spectrum of $L$, denoted $Sp(L)$, is contained in the Gershgorin discs. Thus, the eigenvalues of $L$ are either zero ($\lambda_1$ in this representation) or have a real part strictly positive. }
  \label{fig:gershgorin}
\end{figure}

\subsection{Strongly connected networks ($L$ irreducible)}

We first study the case where $L$ is assumed to be irreducible (see Fig.~\ref{fig:four_cases} {\bf A} and {\bf B}).  In terms of the network $G$ on which the agents interact, this means that $G$ is \textit{strongly connected}.  Intuitively we are assuming that given any two agents $i$ and $j$, that agent $i$ exerts influence over agent $j$ (either directly or indirectly) and vice versa.  Formally, in terms of $L$, we introduce the following definitions.

\begin{definition}
  Let the set of $N$ agents be given by $A$. We say that agent $i$ is influencing agent $j$ if there exists a path $q_1,\,\dots,,q_k\subseteq A$ such that:
  \begin{equation}
    \label{eq:path}
    q_1=i,\quad q_k=j,\quad \text{and } a_{q_i,q_{i+1}} >0 \text{ for all } 1\leq i \leq k-1.
  \end{equation}
  and we write $i \sim j$.
\end{definition}

In other words, $L$ is irreducible if and only if for any $i\neq j$ we have both $i\sim j$ and $j\sim i$.

\begin{definition}
  A graph $G$ is said to be {\it strongly connected} if for any vertices $i$ and $j$:
  \begin{equation}
    \label{eq:strongly_connect}
    i\sim j \quad \text{\bf and} \quad j\sim i.
  \end{equation}
  A graph $G$ is said to be {\it weakly connected} if for any vertices $i$ and $j$:
  \begin{equation}
    \label{eq:weakly_connect}
    i\sim j \quad \text{\bf or} \quad j\sim i.
  \end{equation}
\end{definition}

The information about the eigenvalues of $L$ given in Proposition~\ref{ppo:gershgorin} is not sufficient to characterize the long-term behavior of the model.  We need to investigate the algebraic multiplicity of the zero eigenvalue of $L$.  Since we are assuming that $L$ is irreducible we may leverage the Perron-Frobenius theorem to deduce that this eigenvalue is in fact simple.

\begin{lemma}\label{lem:simple}
  If $L$ is irreducible then the eigenvalue $0$ is simple.
\end{lemma}
\begin{proof}
  By Proposition \ref{ppo:gershgorin} we have that the real parts of the eigenvalues of $L$ are all nonnegative and there are no purely imaginary eigenvalues.  We also note that 0 is an eigenvalue of L of at least multiplicity one because it is associated to ${\bf 1}$. Consider the matrix $D = L-cId$ where $c\in \R$ and notice that $\lambda = a + ib$ is an eigenvalue of $L$ if and only if $\widehat{\lambda} = a - c + ib$ is an eigenvalue of $D$. Therefore, since 0 is an eigenvalue of $L$ we must have that $-c$ is an eigenvalue of $D$. If we choose $c$ such that:
  \begin{align*}
      c>\max\limits_{i}\frac{b_{i}^{2}+a_{i}^{2}}{2a_{i}},
  \end{align*}
  we have that:
  \begin{align*}
      \sqrt{(a_{i}-c)^{2}+b_{i}^{2}}<c\quad\text{for all $i$}.
  \end{align*}
  Therefore $-c$ is the eigenvalue of largest modulus of $D$ and since $D$ is irreducible we have by the Perron-Frobenius theorem that $-c$ must be simple and since the eigenvalues of $D$ are in a one to one correspondence with the eigenvalues of $L$ we must have that 0 is a simple eigenvalue of $L$ as desired.
\end{proof}

Using the properties of $L$ found above we can now exploit the following general fact from the theory of ordinary differential equations to deduce that the model converges to a consensus unconditionally in the case where $L$ is irreducible.  This result will prove useful for showing the convergence of the model and studying consensus in the case of an arbitrary choice for $L$ as well.

\begin{lemma}\label{lem:ode}
  Given a linear system defined by
  \begin{displaymath}
    \vect{x'}=A\vect{x}, \quad\vect{x}(0)=\vect{x}_{0}.
  \end{displaymath}
  Assume $A$ has a zero eigenvalue of multiplicity $m$ with $m$ linearly independent associated eigenvectors and every non-zero eigenvalue $\lambda$ of $A$ satisfies $Re(\lambda)<0$. Then
  \begin{displaymath}
    \lim_{t\rightarrow\infty}\vect{x}(t)=\vect{u}
  \end{displaymath}
  where $\vect{u}$ is in the center subspace of $A$.  
\end{lemma}

\begin{proof}
 See Appendix \ref{sec:cv_linear}.
\end{proof}


Since we have shown that in the case where $L$ is irreducible that its zero eigenvalue is simple, we obtain immediately the convergence of the consensus model to a consensus in this case.

\begin{theorem}
  If $L$ is irreducible then the consensus model converges to a consensus unconditionally.
\end{theorem}
\begin{proof}
  By Proposition \ref{ppo:gershgorin} we know that the eigenvalues of of $-L$ satisfy Re$(\lambda_{i})\leq 0$.  By Lemma \ref{lem:simple} we know that $0$ is a simple eigenvalue of $-L$ associated to ${\bf 1}$ and therefore by
  Lemma \ref{lem:ode} we can deduce that the model converges to a consensus as the center subspace of $L$ is spanned by ${\bf 1}$.
\end{proof}

\subsubsection{Connected undirected networks ($L$ symmetric)}

We now study a special case of when $L$ is irreducible by adding the assumption that $L$ is also symmetric.  Here we are implicitly assuming that every direct interaction between individuals is mutual.  We know from the last section that a consensus will be reached unconditionally, however in this case we will show that the opinion the agents convene on is the \textit{initial average of opinions}. The average opinion is defined as:
\begin{equation}
  \label{eq:bar_s}
  \bar{s}(t) := \frac{1}{N}\sum_{i=1}^{N}s_{i}(t).
\end{equation}
We first show that the average opinion is conserved by the consensus model in this case.

\begin{lemma}\label{lem:symmetric_average}
  If $L$ is symmetric then the average opinion $\bar{s}(t)$ \eqref{eq:bar_s} is preserved by the consensus model.
\end{lemma}
\begin{proof}
  Using the symmetry $a_{ij}=a_{ji}$, we find:
  \begin{align*} \bar{s}\:'(t)\:=\frac{1}{N}\sum\limits_{i=1}^{N}s'_{i}(t)=\frac{1}{N}\sum\limits_{i,j}a_{ij}(s_{j}(t)-s_{i}(t))=0.
  \end{align*}

\end{proof}
We can now show that the opinion the agents convene on is the average initial opinion and that the convergence is exponential with rate at least $\lambda_{2}$, the second smallest eigenvalue of $L$, also known as the \textit{Fieldler number} or \textit{algebraic connectivity} of the graph $G$.

\begin{remark}
  \label{rem:balanced}
We will point out that we can impose a weaker condition than symmetry on $L$ and still maintain the conservation of the average opinion (and therefore convergence to a consensus that is the average opinion).  Recall that if $L$ is not symmetric then the graph associated to $L$ is a directed graph.  If we assume that the graph is \textit{balanced}, i.e. that $L$ satisfies:
\begin{align}\label{eq:balanced}
  \sum_{j=1,\,j\neq i}^N a_{ij} = \sum_{j=1,\,j\neq i}^Na_{ji} \qquad \text{for all } i,
\end{align}
then ${\bf 1}$ must be a left eigenvector of $L$ associated to zero as well (equivalently a right eigenvector of $L^{T}$).  In this case the average opinion is conserved as:
\begin{align}
\bar{s}(t)\:' = \langle -Ls, {\bf 1}\rangle = \langle{s, -L^{T}{\bf 1}}\rangle = 0,
\end{align}
using \eqref{eq:balanced}.
\end{remark}
However, in the following corollary we will maintain the assumption that $L$ is symmetric in order to guarantee that it can be diagonalized.

\begin{corollary}
  \label{cor:symmetric_consensus}
  Suppose $L$ is irreducible and symmetric. Then the solution of the consensus model, ${\bf s}(t)$, satisfies:
  \begin{align*}
    \vect{s}(t)\rightarrow \bar{s}(0)\:{\bf 1} \quad\text{as}\quad t\rightarrow +\infty.
  \end{align*}
  with ${\bar s}$ the average opinion \eqref{eq:bar_s}. Moreover, $|s_{i}(t)-\bar{s}(0)|\leq Ce^{-\lambda_{2}t}$ where C depends only on the initial condition and $\lambda_{2}$ is the second largest eigenvalue of $L$.
\end{corollary}

\begin{proof}
  The consensus model is a linear system, therefore its solution is given by:
  \begin{align*}
    \mathbf{s}(t) = e^{-tL}\mathbf{s_{0}}
  \end{align*}
  where $\mathbf{s_{0}} = (s_{1}(0),...,s_{N}(0))^{T}$. Note that since ${\bf 1}$ is an eigenvector of $L$ corresponding to the eigenvalue $0$, it is also an eigenvector of $ e^{-tL}$ corresponding to the eigenvalue of $1$.  Therefore we may write:
  \begin{equation}\label{eq:eigenrewrite}
    \mathbf{s}(t) - \bar{s}(0){\bf 1} = e^{-tL}(\mathbf{s_{0}}-\bar{s}(0)\:{\bf 1}).
  \end{equation}
  Also note that since $L$ is diagonal dominant and symmetric that there exists $P$ such that
  \begin{equation}\label{eq:decomp}
    L = PDP^{-1}
  \end{equation}
  where $D = \text{diag}(0, \lambda_{2}, \lambda_{3},...)$ and $P$ is the matrix composed of the eigenvectors of $L$.  Since $L$ is symmetric these eigenvalues are real.
  By Proposition \ref{ppo:gershgorin} and Lemma \ref{lem:simple} there is exactly one zero eigenvalue and $\lambda_{i}$ is strictly positive for $2\leq i\leq n$. We now define:
  \begin{align*}
    \mathbf{y}(t):=P^{-1}(\mathbf{s}(t)-\bar{s}(0)\:{\bf 1}),
  \end{align*}
  this is the difference between the opinion at time t and the average opinion represented in the diagonal coordinate system. Notice that:
  \begin{align*}
    \mathbf{y}'(t) &= \frac{d}{dt}[P^{-1}(\mathbf{s}(t) - \bar{s}(0)\:{\bf 1})] = D\mathbf{y}(t).
  \end{align*}
  Therefore, since this is an uncoupled system of linear equations we must have that:
  \begin{equation*}
    y_{i}(t) = y_{i}(0)e^{-\lambda_{i}t}.
  \end{equation*}
  So for $i\geq 2$ we must have that $y_{i}(t)\rightarrow 0$ as $t\rightarrow +\infty$ exponentially with rate at least $\lambda_{2}$.  To conclude, it remains to show that $y_1(t) = 0$.
  Using that the eigenvectors of $L$ form an orthogonal basis, the entry $y_1(t)$ is given by:
  \begin{displaymath}
    y_1(t) = \langle{\bf s}(t) - \bar{s}(0)\:{\bf 1},\, \frac{{\bf 1}}{\|{\bf 1}\|}\rangle = \bar{s}(t) - \bar{s}(0)= 0
  \end{displaymath}
  since the mean value $\bar{s}(t)$ is preserved over time.
\end{proof}

\subsection{Weakly connected and disconnected networks (arbitrary $L$)}

In this section we remove the assumption that $L$ is irreducible.  In terms of the network associated to $L$, this translates to assuming that the network on which the agents interact is merely \textit{weakly connected} or disconnected.  The case of weakly connected or disconnected graphs is more delicate. In this case we are implicitly assuming that individuals, or communities of individuals, may be isolated from influence from the network.  
Under these assumptions on $L$ a consensus might not always be reached. See for example the right plot in Figure \ref{fig:evolution_2_cases}.
However we will show that the dynamics always converge and provide a condition on $L$ weaker then irreducibility that is equivalent to unconditional convergence to a consensus.  Our main tool for studying the dynamics in this case is Lemma \ref{lem:ode} and the decomposition of $L$ according to its strongly connected components (similar to the so-called \textit{Frobenius normal form}).  


Given a weakly connected or disconnected graph we may partition its set of vertices into \textit{strongly connected components}.  Two vertices $i$ and $j$ are in the same strongly connected component if there exists a directed path from $i$ to $j$ and vice versa (i.e. $i \sim j$ and $j \sim i$).  If we treat each strongly connected component as a vertex of a new graph $G'$, and notice that this graph is necessarily directed acyclic, then by relabeling the vertices of $G$ in a way that agrees with the topological ordering of the vertices in $G'$ we can represent $L$ in the following form (see Appendix \ref{sec:frobenius} and Figure \ref{fig:relabeling}).
\begin{equation}
  \label{eq:frobenius_normal_form}
  L=
  \begin{bmatrix}
    B_{1}                                    \\
    &  \ddots             &   & \text{\huge0}\\
    &               & B_{k}                \\
    & \text{\huge*} &   & \ddots            \\
    &               &   &   & B_{d}
  \end{bmatrix}.
\end{equation}
 Each block, $B_{k}$, on the diagonal is irreducible as they correspond to strongly connected components of $G$. Notice that $d$ denotes the number of strongly connected components, thus when $L$ is irreducible we have $d=1$. We give an example of such a relabeling in Figure \ref{fig:relabeling}.

 \begin{figure}[ht]
  \centering
  \includegraphics[scale=.4]{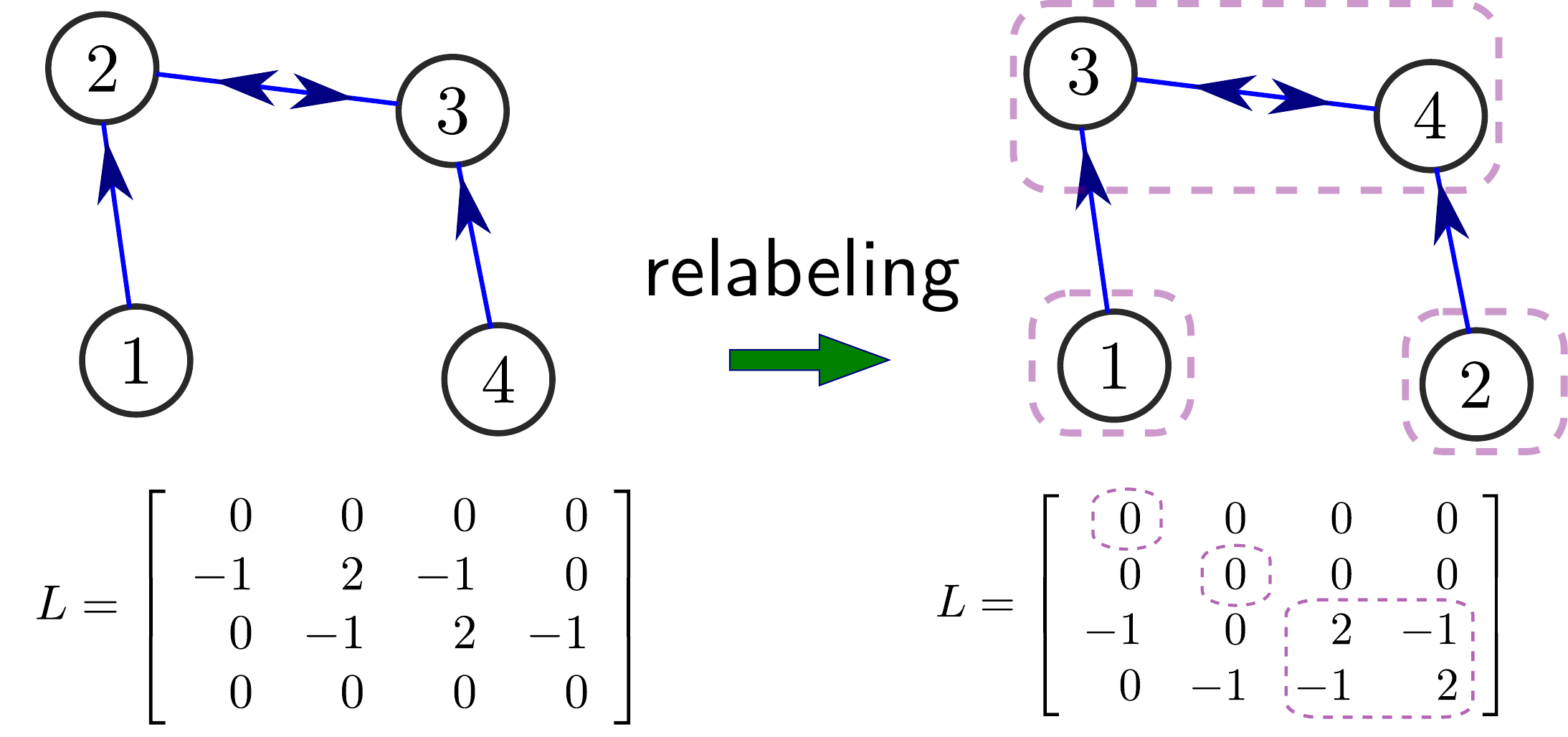}
  \caption{An example of relabeling to obtain the Frobenius normal form.}
  \label{fig:relabeling}
\end{figure}



Notice that by the definition of $L$ each block $B_{k}$ is diagonal dominant:
\begin{equation}
  \label{eq:diag_dominant}
  b_{ii} \geq \sum_{j,j\neq i} -b_{ij},
\end{equation}
where $b_{ij}$ are the coefficients of the matrix $B_k$. We can further distinguish between blocks.

\begin{definition}
  Assume $L$ is given in the Frobenius normal form \eqref{eq:frobenius_normal_form}. A block $B = (b_{ij})$  is called {\it isolated} if it satisfies:
  \begin{equation}
    \label{eq:isolated}
    b_{ii} = \sum_{j,j\neq i} -b_{ij} \qquad \text{for all } i.
  \end{equation}
  In other words, vertices from the block $B_k$ are only influenced by vertices in block $B_k$:
  \begin{equation}
    \label{eq:isolated_2}
    \hspace{-2cm}    \text{\bf isolated}: \hspace{1.2cm} \text{if } i \in B_k \text{ and } j \notin B_k , \text{ then } j \not \sim i.
  \end{equation}
  Note however that $i$ might influence $j$ (i.e. $i\sim j$).  A non-isolated block $B_k$ is called {\it absorbing} if the vertices from block $B_k$ can only influence vertices in $B_k$:
  \begin{equation}
    \label{eq:absorbing}
    \hspace{-2cm}\text{\bf absorbing}: \hspace{1.2cm}     \text{if } i \in B_k \text{ and } j \notin B_k , \text{ then } i \not \sim j.
  \end{equation}
  Note again that $j$ might influence $i$ (i.e. $j\sim i$).
\end{definition}

Notice that a block can be neither absorbing nor isolated. In the example from Figure \ref{fig:relabeling}, the matrix $L$ is composed of two isolated blocks and one absorbing block and has no blocks that are neither.  Isolated blocks of $L$ correspond to \textit{isolated strongly connected components} of the graph $G$ on which the agents interact. We now establish a correspondence between isolated blocks and zero eigenvectors of $L$.

\begin{proposition}
  \label{ppo:zero}
  If $L$ has $m$ isolated blocks there exist at least $m$ linearly independent eigenvectors of $L$ associated to the zero eigenvalue.
\end{proposition}

\begin{proof}
  Assume that $B_{k}$ is an isolated block of $L$ and that $B_{k}$ is $l\times l$. Then, the vector ${\bf 1}=(1,1,...,1)^{T}$ of length $l$ is a right eigenvector of $B_{k}$ associated to the zero eigenvalue.  Since $B_{k}$ is irreducible and is isolated we have by the Perron-Frobenius theorem that ${\bf 1}$ is the only zero right eigenvector of $B_{k}$ (see the proof of Lemma \ref{lem:simple}).  Since for any matrix the collection of left and right eigenvalues as well as the dimensions of the corresponding eigenspaces are equal, there must exist a unique \textit{left} eigenvector of $B_{k}$ associated to zero denoted $\bar{{\bf u}}$.  
  We then deduce a {\it left} eigenvector of $L$ in the form:
  \begin{displaymath}
    \vect{w}=(0,\dots,0,\,\bar{{\bf u}},0,\dots,0)
  \end{displaymath}
  where the first entry of $\bar{{\bf u}}$ is the $i$th entry of $\vect{w}$. 
  We could clearly repeat this argument for all isolated blocks to produce a collection of left eigenvectors of $L$ associated to zero which are necessarily linearly independent as they are all nonzero in disjoint coordinates.  Therefore, since the dimensions of the left and right eigenspaces of $L$ are equal, if there are $m$ isolated blocks of $L$ there must be $m$ linearly independent right eigenvectors of $L$ associated to zero.
\end{proof}

So, if $L$ has $m$ isolated blocks then it has at least $m$ linearly independent zero eigenvectors.  We will eventually show that these are the only zero eigenvectors of $L$ and deduce convergence by Lemma \ref{lem:ode}; as a first step we show that non isolated blocks of $L$ do not have a zero eigenvalue.

\begin{proposition}
  \label{ppo:no_zero_eigenvector}
  Suppose that $B_k$ is a non-isolated block. Then $B_k$ does not have a zero eigenvalue.
\end{proposition}

\begin{proof}
  Denote by $b_{ij}$ the components of $B_k$. Since $B_k$ is a non-isolated block, there exists $i_0$ such that:
  \begin{displaymath}
    b_{i_0 i_0} > \sum_{j,j\neq i_0} -b_{i_0\,j}.
  \end{displaymath}
  In particular, the constant vector ${\bf 1}=(1,\dots,1)$ cannot be a zero-eigenvector of $B_k$.  Suppose for the sake of contradiction that $B_{k}$ had a zero eigenvector, $\vect{u}$, and let $u_{i}$ be the largest entry of $\vect{u}$ i.e. $u_i=|{\bf u}|_\infty$ (we assume without loss of generality that $u_i>0$).  Then we have:
  \begin{eqnarray*}
    b_{ii}u_{i}+\sum_{j\neq i}b_{ij}u_{j}\; =\; 0 \quad &\Rightarrow& \quad     (-\sum_{j\neq i}b_{ij})u_{i}+\sum_{j\neq i}b_{ij}u_{j} \;\leq\; 0    \\
    &\Rightarrow& \quad \sum_{j\neq i}b_{ij}(u_{j}-u_i) \;\leq\; 0.
  \end{eqnarray*}
  Since $b_{ij}\leq 0 $ (for $j\neq i$) and $u_j-u_i\leq 0$,  we must have $u_j=u_i$ if $b_{ij}<0$. In other words, the coefficients of the eigenvector ${\bf u}$ are constant on the indices connected to $i$:
  \begin{displaymath}
    u_j=|{\bf u}|_\infty \quad \text{if} \quad i\sim j.
  \end{displaymath}
  Iterating the argument, we deduce that $u_j=|{\bf u}|_\infty$ if there exists a path joining $i$ to $j$. Since $B_k$ is irreducible we deduce that $u_j=|{\bf u}|_\infty$ for all $j$. Therefore, ${\bf u}=(|{\bf u}|_\infty,\dots,|{\bf u}|_\infty)$. This is a contradiction as ${\bf 1}$ is not a zero eigenvector of $B_k$.


\end{proof}

Combining propositions \ref{ppo:zero} and \ref{ppo:no_zero_eigenvector} leads to the following theorem.

\begin{theorem}\label{reducible_convergence}
  The consensus model \eqref{model_agent} converges (but not necessarily to a consensus).
\end{theorem}

\begin{proof}
  Using the decomposition of $L$ into strongly connected components (see eq. \eqref{eq:frobenius_normal_form} and appendix \ref{sec:frobenius}), we can write the spectrum of $L$ as:
  \begin{equation}
    \text{spec}(L) = \bigcup_{k=1}^d\text{spec}(B_{k}) \label{eq:spec}.
  \end{equation}
  where $B_{k}$ corresponds to a strongly connected component.  Only isolated blocks $B_k$ have zero eigenvalues and moreover $L$ has a zero eigenvalue of multiplicity equal to the number of isolated blocks (denoted $m$) by (\ref{eq:spec}) and Propositions \ref{ppo:no_zero_eigenvector} and \ref{ppo:zero}. Moreover, Proposition \ref{ppo:zero} states that there exist $m$ linearly independent eigenvectors of $L$ associated with the zero eigenvalue and therefore we deduce the convergence by Lemma \ref{lem:ode}.

\end{proof}

\begin{remark}

In the proof of Lemma \ref{lem:ode} (see appendix \ref{sec:cv_linear}) we find that we can bound the convergence of $\vect{s}(t)$ in terms of the spectral gap of $L$, denoted $\lambda_{2}$, some $0<\epsilon<\text{Re}(\lambda_{2})$ and some $C>0$:
\begin{align}\label{rate_of_convergence}
\norm{\vect{s}(t)} \leq Ce^{-(\text{Re}(\lambda_{2})+\epsilon)t}.
\end{align}
If $L$ is symmetric, $\lambda_{2}$ is known as the \textit{algebraic connectivity} of the graph $G$ and is a measure of how well connected $G$ is.  In this context the bound (\ref{rate_of_convergence}) can be interpreted as stating that the consensus model converges faster if the network on which agents interact is better connected.
\end{remark}

We can observe from Theorem \ref{reducible_convergence} that the number of zero eigenvectors of $L$ is in correspondence with the number of isolated blocks of $L$.  This observation supplies us with a condition on $L$ that is equivalent to unconditional convergence to a consensus.

\begin{corollary}\label{cor:unconditional_consensus}
  The consensus model converges to a consensus unconditionally if and only if $L$ has exactly one isolated block.
\end{corollary}

\begin{proof}
  The vector ${\bf 1}$ is always a zero eigenvector of $L$.  Since $L$ has only one isolated block, the eigenvalue $0$ must be simple by Propositions \ref{ppo:zero} and \ref{ppo:no_zero_eigenvector}. Therefore, ${\bf 1}$ is the only zero eigenvector of $L$ and therefore Lemma \ref{lem:ode} implies that the dynamics converge to a consensus as the center subspace of $L$ is spanned by ${\bf 1}$.
\end{proof}

\begin{remark}

Recall that isolated blocks of $L$ correspond to isolated blocks of the graph $G$ on which the agents interact.  We can interpret an isolated block as "leading" the rest of the network.  Given a network with only one isolated block the consensus model implies that nodes in an isolated block receive no influence from nodes outside of the isolated block and since we know by Corollary \ref{cor:unconditional_consensus} that the entire network reaches a consensus, we deduce that the consensus converged to by the whole network is the consensus converged to by agents in the isolated block. The network can fail to converge to a consensus if there are multiple isolated blocks that converge on different opinions (See the right plot in Figure \ref{fig:evolution_2_cases}) - this clearly depends on the initial opinions of the agents in the network.  Thus, if a network has multiple isolated blocks the consensus model does \textit{not} converge to a consensus unconditionally (although a consensus might still be reached if all isolated blocks happen to converge on the same opinion).
\end{remark}

\subsection{Consensus as a control problem}
We now slightly shift our perspective to consider a question inspired by control theory.  Given $N$ agents with initial opinions $\set{s_{1}(0),...,s_{N}(0)}$ and a desired opinion $s^{*}\in \Rn$ that one would like every agent to converge on, we investigate if there is a way to connect the agents in a network so that under the consensus model we have:
\begin{align*}\lim_{t\rightarrow\infty}\vect{s}(t) = s^{*}{\bf 1}.\end{align*}
We can use our observations about how isolated blocks affect the behavior of the consensus model to answer this question. For ease of notation we denote:
\begin{align*}
  s_{max} &:= s_{i}\quad \text{where}\quad |s_{i}| = \text{max}\set{|s_{1}(0)|,...,|s_{N}(0)|},\\
  s_{min} &:= s_{j}\quad \text{where}\quad |s_{j}| = \text{min}\set{|s_{1}(0)|,...,|s_{N}(0)|}.
\end{align*}
We also note that there is no way to connect the agents to achieve the desired result if $s^{*}\notin \text{Conv}(\set{s_{1}(0),...,s_{N}(0)})$ (the convex hull of initial opinions) as the consensus model forces the convex hull of $\set{s_{1}(t),...,s_{N}(t)}$ to decrease in diameter as time evolves.
\begin{proposition}
  Given $(s_{1}(0),...,s_{N}(0))$ and $s^{*}\in \text{Conv}(\set{s_{1}(0),...,s_{N}(0)})$ there exists $L$ such that:
  \begin{align*}\lim_{t\rightarrow\infty}\vect{s}(t) = s^{*}{\bf 1}.\end{align*}
\end{proposition}
\begin{proof}
  For ease of notation, assume
\begin{align*}
  s_{max} = s_{1}(0) \quad\text{and} \quad s_{min} = s_{2}(0).
\end{align*}
We first consider the two-agent network who's evolution is given by
\begin{align}
  \begin{pmatrix} s_{1}'(t)\\ s_{2}'(t)\\ \end{pmatrix}=\begin{pmatrix}-\beta & \beta \\ \alpha & -\alpha\\   \end{pmatrix} \begin{pmatrix}s_{max}\\s_{min}\\ \end{pmatrix}. 
\end{align}
We find that if we choose $\alpha$ and $\beta$ such that
\begin{align*}
    s^{*} = \frac{1}{\alpha+\beta}(\beta s_{max}+\alpha s_{min}),
\end{align*}
we have
\begin{align*}
    \lim_{t\rightarrow\infty}\begin{pmatrix} s_{1}(t)\\ s_{2}(t)\\ \end{pmatrix} &=\begin{pmatrix} s^{*}\\ s^{*}\\ \end{pmatrix}.
\end{align*}

So, given $N$ agents, if we connect agents 1 and 2 with the weights found above and the rest of the agents in a ``chain'', i.e. we choose an $L$ of the form:
\begin{align*}
  L = \begin{pmatrix}
    \beta & -\beta & 0 & 0 & 0 &\dots & 0 \\
    -\alpha & \alpha & 0 & 0 & 0 &\dots & 0 \\
    0 & -1 &  & 1 & 0 &\dots & 0 \\
    0 & 0 & 0 & -1 & 1 &\dots & 0 \\
    &  &  & \vdots & & & \\
    0 & 0 & 0 & 0 & 0 &\dots & 1 \\
  \end{pmatrix}
\end{align*}
then the network on which the agents interact will only have one isolated block  consisting of agent 1 and agent 2.  Our choice of weights guarantees that the opinions of agents 1 and 2 converge to $s^{*}$ and since they comprise the only isolated block, Corollary \ref{cor:unconditional_consensus} guarantees that we must have that:
\begin{align*}\lim_{t\rightarrow\infty}\vect{s}(t) = s^{*}{\bf 1}\end{align*}
as desired.
\end{proof}

\section{Opinion dynamics: stochastic approach}

\subsection{Introduction}

We now adopt another point of view to model opinion dynamics. Rather than changing their opinions gradually according to their neighbors, agents now ``jump'' their opinion to one of their neighbors. The weights $a_{ij}$ now encode the {\it probability} of agent $i$ switching its opinion to agent $j$. Mathematically, we will use a Poisson process of rate $a_{ij}$ to model this event; opinion $S_i$ ``jumps'' to opinion $S_j$ at rate $a_{ij}$.
\begin{equation}
  \label{eq:jump}
  (S_i,S_j) \stackrel{\text{rate}\; a_{ij}}{\leadsto} (S_j,S_j).
\end{equation}
The dynamics are now a continuous time Markov chain instead of a system of differential equations. We begin by formally defining the dynamics in terms of the stochastic process $S_{i}$ that describes the evolution of the opinion of agent $i$.




\begin{definition}
Consider $N$ agents with opinion $S_i(t)\in \R$, $t\geq 0$. Let  $(a_{ij})_{ij}$ be the entries of the adjacency matrix of the graph $G$ on which the agents interact. For each tuple $(i,j)$ with $i\neq j$, we associate an independent Poisson process $N^{ij}(t)$ with rate $a_{ij}$. The {\bf stochastic consensus model} (SCM) is defined as:
\begin{equation}
  \label{eq:s_i_stochastic}
  \mathrm{d} S_i = \sum_{j=1,j\neq i}^N (S_j-S_i) \mathrm{d}N^{ij}.
\end{equation}
\end{definition}
Again, for the sake of clarity of presentation we restrict opinions to being one dimensional and we note that the  stochastic consensus model can be easily extended for opinions in $\mathbb{R}^n$ and that all results discussed will still hold analogous proofs.  Notice that one can write the Poisson process as $N^{ij}(t) = N(a_{ij} t)$ where $N(t)$ is a Poisson process of rate $1$. Each time a Poisson process $N^{ij}$ increases by one unit, the opinion of agent $i$, $S_i$, ``jumps'' to  $S_j$. We illustrate these dynamics in Figure \ref{fig:jump} on a graph with only three nodes and two links.

\begin{figure}[ht!] \label{fig:jump}
    \centering
    \includegraphics[scale = .5]{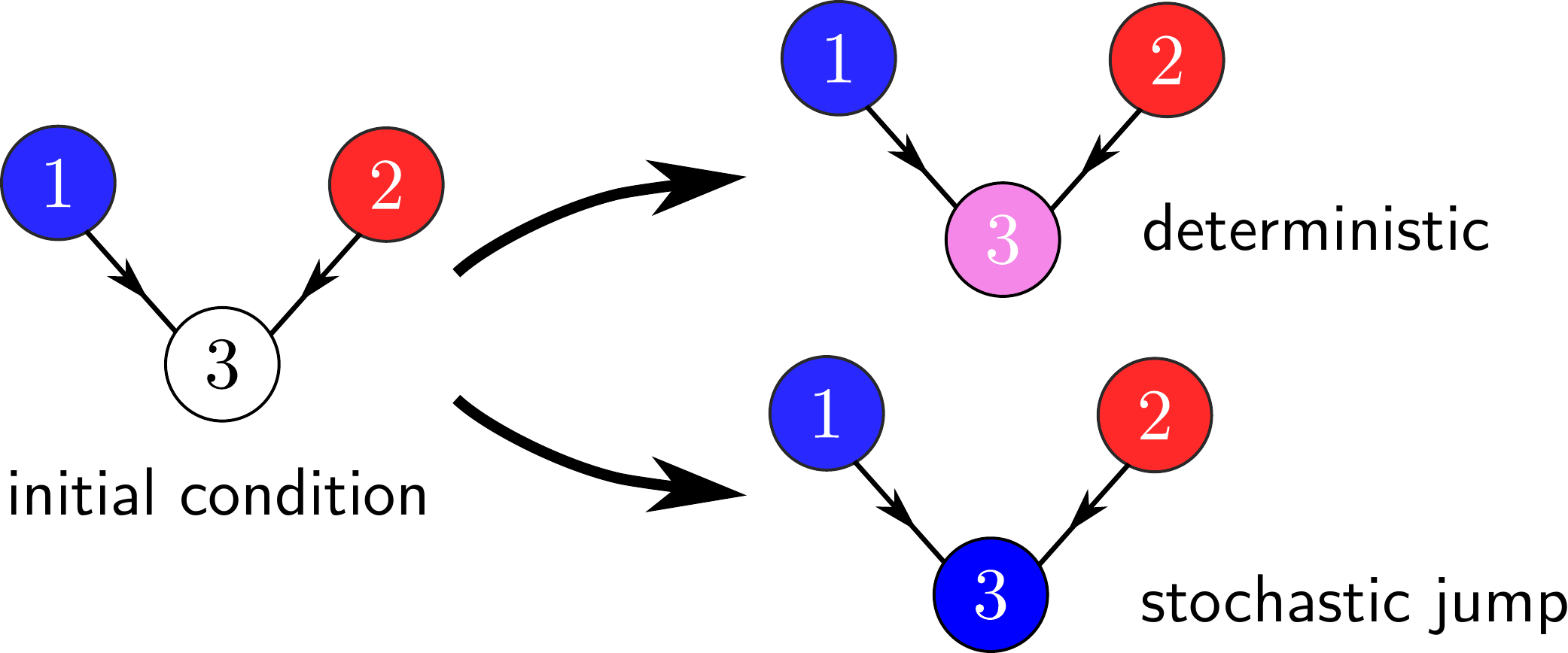}
    \caption{Illustration of the deterministic and stochastic dynamics. The node $3$ receives influence from both node $1$ and node $2$. In the deterministic dynamics \eqref{model_agent}, the opinion of $s_3$ will become a weighted average of opinion $s_1$ and $s_2$. However, in the stochastic dynamics \eqref{eq:s_i_stochastic}, the opinion $S_3$ will continue switching between opinions $S_1$ and $S_2$.}
    \label{fig:deterministic_stochastic}
\end{figure}

The stochastic consensus model can also be written in terms of its generator. For any pair $(i,j)$, denote by $\Phi_{ij}$ the function:
\begin{equation}
  \label{eq:Phi}
  \Phi_{ij}(s_1,\dots,s_i,\dots,s_j,\dots,s_N) \;\; = \;\; (s_1,\dots,s_j,\dots,s_j,\dots,s_N).
\end{equation}
Let ${\bf S}=(S_1,\dots,S_N)$. For any smooth test function $\varphi:\mathbb{R}^N\to\mathbb{R}$, we have:
\begin{equation}
  \label{eq:generator_SDE}
  \mathrm{d} \mathbb{E}[\varphi({\bf S})] = \sum_{1\leq i,j \leq N,\,j\neq i} a_{ij} \mathbb{E}[\varphi(\Phi_{ij}({\bf S})) - \varphi({\bf S})]\,\mathrm{d}t.
\end{equation}

The stochastic model can also be more simply described through its associated embedded Markov chain. Instead of describing the continuous evolution of each opinion $S_i(t)$, we only record the evolution of each jump:
\begin{displaymath}
  S_i^n = S_i(T_n)
\end{displaymath}
where $T_n$ is the time of the $n^\text{th}$ jump. Notice that $T_{n+1}-T_n$ is an exponential random variable with frequency $\sigma$ - the rate of any jump occurring:
\begin{equation}
  \label{eq:sigma}
  \sigma=\sum_{i,j,j\neq i} a_{ij},
\end{equation}
by the superposition property of independent Poisson processes.  The evolution of $S_i^n$ is then given by a discrete Markov process described by the probability transition matrix:
\begin{equation}
  \label{eq:p_ij}
  p_{ij} = P\Big( (S_i^n,S_j^n) \leadsto  (S_j^{n+1},S_j^{n+1})\Big) = \frac{a_{ij}}{\sigma},
\end{equation}
where $a_{ij}$ is the rate of agent $i$ switching to the opinion of agent $j$ and $\sigma$ given by \eqref{eq:sigma}. The generator of this Markov chain can be described as:
\begin{equation}
  \label{eq:generator_Markov}
  \mathbb{E}[\varphi({\bf S}^{n+1})] = \mathbb{E}[\varphi({\bf S}^{n})] + \sum_{i,j,j\neq i} p_{ij} \mathbb{E}[\varphi(\Phi_{ij}({\bf S}^n)) - \varphi({\bf S}^n)].
\end{equation}


We will now investigate the long-time behavior of these dynamics and analyze how the structure of the graph $G$ affects the convergence.

\begin{remark}
  Taking the expectation of the stochastic dynamics \eqref{eq:s_i_stochastic}, we recover the deterministic model \eqref{model_agent}. Indeed, denote
  \begin{equation}
    \label{eq:s_i_bar}
    \bar{s}_i = \mathbb{E}[S_i(t)],
  \end{equation}
  using that $\mathbb{E}[N^{ij}(t)]=a_{ij}\cdot t$ \cite{privault_stochastic_2013} and that $dN^{ij}(t)$ is independent of $S_i(t)$ and $S_j(t)$, we deduce:
  \begin{equation}
    \label{eq:derive_deterministic}
    \mathrm{d} \bar{s}_i = \sum_{j\neq i} (\bar{s}_j-\bar{s}_i) a_{ij} \mathrm{d}t
  \end{equation}
  which corresponds to the dynamics \eqref{model_agent}. Thus, we can deduce the behavior of the deterministic model from the stochastic model using a simple Monte-Carlo method. Conversely, the deterministic model gives information about the average behavior of the stochastic model.  We also notice that in the case that all agents are connected ($a_{ij}=1$ for all $i,j$), we recover a special case of the so-called ``choose-the-leader'' dynamics \cite{carlen_kinetic_2013,carlen_kinetic_2013-1}.
\end{remark}

\subsection{Convergence (in probability)}

Analogous to our discussion of the deterministic consensus model, we are going to prove that the stochastic model converges to a consensus unconditionally if the graph, $G$, has only one isolated strongly connected component. However, the convergence will be much slower than the deterministic model. The key tool is to use the notions of \textit{absorbing states} and \textit{absorbing Markov chains}.




\begin{definition}[Absorbing states and absorbing Markov chain]
  Let $\mathbf{S}^n$ be a Markov chain on a state space $\mathcal{C}$. A state $\mathbf{S} \in \mathcal{C}$ is called \textit{absorbing} if $p_{\mathbf{S,S}} = 1$ -- that is, the probability of staying in $\mathbf{S}$ given that the chain has arrived in $\mathbf{S}$ is 1. The Markov chain $\mathbf{S}^n$ is called an \textit{absorbing Markov chain} if for any starting state $\mathbf{S}^0 \in \mathcal{C}$, the chain can reach an absorbing state in finitely many transitions with positive probability. A Markov chain is \textit{absorbed} if it reaches an absorbing state.
\end{definition}
For instance, in our case, ${\bf 1}=(1,1,..,1)^{T}$ is an absorbing state. We can find sufficient conditions on the network $G$ for the embedded Markov chain ${\bf S}^n$ of the stochastic consensus model to be absorbing.



\begin{theorem}
  \label{thm:cv_consensus_SCM}
  Assume that the adjacency matrix $L=(a_{ij})_{ij}$ has only one isolated strongly connected component. Then the stochastic consensus model (SCM) converges with probability $1$ to a consensus.
\end{theorem}
\begin{proof}
  We are going to show  the convergence of the embedded Markov chain $\{{\bf S}^n\}_n$ from which we can easily deduce the  convergence of the continuous process $\{{\bf S}(t)\}_{t\geq 0}$.

  Notice first that for a given initial condition ${\bf S}^0$, the embedded Markov chain ${\bf S}^n$ evolves in a finite space $\mathcal{C}$ of size $N^N$. 
  Indeed, each component ${\bf S}_i^n$ has to be one of the initial opinion components $\big({\bf S}_i^0\big)_{1\leq i \leq N}$. Thus, if we can show that $\{{\bf S}^n\}_n$ is an absorbing Markov chain, then ${\bf S}^n$ will be absorbed with probability $1$ [Theorem 11.3, \cite{Grinstead}].

  Consider now any opinion $S_i^0$ in the isolated strongly connected component of the graph. By assumption on the adjacency matrix $L$, there exists for any $j$ a path joining $i$ to $j$ \eqref{eq:path}. Thus, there is a strictly positive probability (i.e. $p_{i,q_2}\dots p_{q_{k-1},j}$ where $q_2,\dots,q_{k-1}$ path joining $i$ to $j$) such that $S_j^k=S_i^k=S_i^0$ after $k$ steps. Iterating the argument for all $j$, we find a non-zero probability that $S_j^n=S_i^0$ for all $j$ with $n$ finite (with a rough upper-bound of $N^N$). Thus, we conclude that $\{{\bf S}^n\}_n$ is an absorbing Markov chain.
\end{proof}

\begin{remark}
  From Theorem \ref{thm:cv_consensus_SCM}, we can deduce a simple proof for corollary \ref{cor:unconditional_consensus}  concerning the convergence of the deterministic model. 
  Indeed, for any $\varepsilon>0$ and initial condition ${\bf S}^0$, there exists a time $T$ such that for any $t \geq T$ we have $\mathbb{P}\Big({\bf S}(t) \text{ is a consensus}\Big)\geq1-\varepsilon$.
  Taking the expectation of the stochastic process, ${\bf s}(t) = \mathbb{E}[{\bf S}(t)]$, we deduce that ${\bf s}(t)$ converges to the expected consensus. However notice that we lose the rate of convergence as we did not analyze the eigenvalues of the dynamics.
\end{remark}

\subsection{Decay rate for undirected graphs}

Theorem \ref{thm:cv_consensus_SCM} gives a necessary and sufficient conditions for the  unconditional convergence to consensus of the stochastic model. However, we do not have any information about the rate of convergence of the process. In the deterministic case, we have seen that the convergence is exponential with a explicit rate given in terms of the second largest eigenvalue of the matrix $L$, $\lambda_2$. We would like to investigate if one could obtain a similar explicit convergence rate for the stochastic model. 
With this aim, we are going to investigate the empirical variance of the process:
\begin{equation}
  \label{eq:V_var}
  V(\mathbf{\bf S}(t)) = \frac{1}{N} \sum_{i=1}^N |S_i(t) - \bar{S}(t)|^2 = \frac{1}{2N^2}\sum_{1\leq i,j \leq N} |S_i(t) - S_j(t)|^2,
\end{equation}
where $\bar{S}(t)$ is the empirical mean:
\begin{equation}
  \label{eq:mean_S}
  \overline{S}(t) = \frac{1}{N} \sum_{i=1}^N S_i(t).
\end{equation}
We will need to assume that the interactions are symmetric, i.e. that $L$  is symmetric.

\begin{lemma}
  Suppose that the adjacency matrix $L$ is symmetric, then the empirical variance $V(\mathbf{S})$ \eqref{eq:V_var} decays according to
  \begin{equation}
    \label{eq:decay_emp_var}
    \frac{d}{dt} \mathbb{E}[V(\mathbf{S})] = -\frac{1}{N^2}\mathbb{E}[\sum_{i,j}a_{ij}|S_j-S_i|^2].
  \end{equation}
\end{lemma}
\begin{proof}
  Using the generator of the stochastic process \eqref{eq:generator_SDE}, we find:
  \begin{align*}
    \frac{d}{dt}\mathbb{E}[V(\mathbf{S})] = \sum_i\sum_{j\neq i}a_{ij}\mathbb{E}[V(\Phi_{ij}(\mathbf{S})) - V(\mathbf{S})].
  \end{align*}
  Moreover, notice the following equality:
  \begin{align*}
    V(\Phi_{ij}(s)) = V(s) + \frac{2}{2N^2}\sum_{k=1}^N \Big(|S_k-S_j|^2-|S_k-S_i|^2\Big) - \frac{2}{2N^2} |S_j-S_i|^2.
  \end{align*}
  Indeed, when $i$ jumps to opinion $j$ (i.e. $\Phi_{ij}({\bf S})$), all the '$(i,k)$' index become  '$(j,k)$' but the difference between opinion $i$ and $j$ is zero.  We then deduce:
  \begin{align*}
    V(\Phi_{ij}(\mathbf{S})) + V(\Phi_{ji}(\mathbf{S})) = 2V(\mathbf{S}) - \frac{4}{2N^2}|S_j-S_i|^2.
  \end{align*}
  Furthermore, exploiting the symmetry $a_{ij} = a_{ji}$ we get that
  \begin{align*}
    a_{ij}V(\Phi_{ij}(\mathbf{S})) + a_{ji}V(\Phi_{ji}(\mathbf{S})) \;\;=\;\; (a_{ij}+a_{ji})\big[V(\mathbf{S}) - \frac{2}{2N^2}|S_j-S_i|^2\big].
  \end{align*}
  Thus
  \begin{align*}
    \frac{d}{dt}\mathbb{E}[V(\mathbf{S})] &= \sum_i\sum_{j\neq i}\mathbb{E}[a_{ij}(V(\Phi_{ij}(\mathbf{S})) - V(\mathbf{S}))]\\
                                          &=\mathbb{E}\big[\sum_{i<j} \big(a_{ij}V(\Phi_{ij}(\mathbf{S})) + a_{ji}V(\Phi_{ji}(\mathbf{S})) - (a_{ij} + a_{ji})V(\mathbf{S})\big)\big]\\
                                          &= -\frac{1}{N^2}\mathbb{E}\big[\sum_{i<j}(a_{ij}+a_{ji})|S_j-S_i|^2\big] = -\frac{1}{N^2}\mathbb{E}\big[\sum_{i,j}a_{ij}|S_j-S_i|^2\big]
  \end{align*}
\end{proof}


We can now deduce a lower bound for the previous equality \eqref{eq:decay_emp_var} in terms of $\lambda_2$, the second smallest eigenvalue of the Laplacian matrix and the algebraic connectivity of the graph $G$.

\begin{lemma}
  Suppose that the Laplacian matrix $L$ of the network $G$ is symmetric and let $\lambda_2$ be its second smallest eigenvalue. Then:
  \begin{equation}
    \label{eq:lower_bound_ld2}
    \sum_{1\leq i,j \leq N,\,j\neq i} a_{ij}|s_j-s_i|^2 \geq 2\lambda_2N\cdot V(\mathbf{S}).
  \end{equation}
  with $V(\mathbf{S})$ the variance \eqref{eq:V_var}.
\end{lemma}
\begin{proof}
  From the above, we note that
  \begin{align*}
    \langle L\mathbf{S},\mathbf{S} \rangle &= \sum_{i=1}^N \Big(\sum_{j=1}^N a_{ij}(S_i-S_j)\Big) S_i = \sum_{ij} a_{ij}(S_i-S_j)S_i\\
                                           &= \sum_{ji} a_{ji}(S_j - S_i)S_j = \sum_{ji} a_{ij}(S_j - S_i)S_j
  \end{align*}
  where we use the symmetry of $L$. Therefore we can write
  \begin{align*}
    \langle L\mathbf{S},\mathbf{S} \rangle &= \frac{1}{2}\Big(\sum_{ij} a_{ij}(S_j - S_i)S_j + \sum_{ij} a_{ij}(S_j - S_i)S_i\Big)\\
                                           &= \frac{1}{2}\sum_{ij}a_{ij}(S_i - S_j)(S_i - S_j) = \frac{1}{2}\sum_{ij}a_{ij}|S_i-S_j|^2.
  \end{align*}
  Notice also that the empirical variance can be written in two ways:
  \begin{displaymath}
    \frac{1}{2N^2} \sum_{ij}|S_i - S_j|^2 = \frac{1}{N} \sum_{i}|S_i - \overline{S}|^2 =  \frac{1}{N} \|\mathbf{S} - \overline{S}{\bf 1}\|^2,
  \end{displaymath}
  where $\overline{S}$ is the average opinion \eqref{eq:mean_S}. Since $L {\bf 1} = 0$, we also deduce:
  \begin{align*}
    \langle L(\mathbf{S}-\overline{S}{\bf 1}),  \mathbf{S} - \overline{S}{\bf 1}\rangle &= \langle L\mathbf{S}, \mathbf{S}\rangle.
  \end{align*}
  Decomposing the symmetric matrix $L$ as $L=P D P^T$ with $D=\text{diag}(0,\lambda_2,\dots,\lambda_N)$, we deduce:
  \begin{align*}
    \langle L\mathbf{S},\mathbf{S}\rangle &= \langle L(\mathbf{S}-\overline{\mathbf{S}}),  \mathbf{S} - \overline{\mathbf{S}}\rangle\\
                                          &= \langle DP^T(\mathbf{S}-\overline{\mathbf{S}}), P^T(\mathbf{S}-\overline{\mathbf{S}})\rangle
                                          = \langle D {\bf w}, {\bf w}\rangle
  \end{align*}
  with $\mathbf{w} = P^T(\mathbf{s}-\overline{\mathbf{s}})$. Thus,
  \begin{align*}
    \langle L\mathbf{S},\mathbf{S}\rangle &= \sum_{i=1}^N \lambda_i |w_i|^2 =\sum_{i=2}^N \lambda_i |w_i|^2 \geq \lambda_2 \sum_{i=2}^N |w_i|^2.
  \end{align*}
  Notice that the first component $\mathbf{w}_1$ has to be zero as the first row of $P^{T}$ is $\alpha\bf{1}$ for some $\alpha \in \R$. 
  \begin{displaymath}
    \lambda_2\sum_{i=2}^N |w_i|^2 = \lambda_2\sum_{i=1}^N |w_i|^2 = \lambda_2\|{\bf w}\|^2.
  \end{displaymath}
  Since the matrix $P$ preserves the norm, we have:
  \begin{align*}
    \langle L\mathbf{S},\mathbf{S}\rangle &\geq \lambda_2 \|{\bf w}\|^2 = \lambda_2 \|S-\overline{S}{\bf 1}\|^2 = \lambda_2 N\cdot V({\bf S}).
  \end{align*}
  Thus
  \begin{align*}
    \sum_{ij}a_{ij}|S_j-S_i|^2 = 2\langle L\mathbf{S},\mathbf{S}\rangle \geq 2 \lambda_2 N V(\mathbf{S}).
  \end{align*}
\end{proof}

Using these tools we can now characterize the convergence rate of the stochastic consensus model on a symmetric network.
\begin{theorem}
  Assume $L$ is symmetric and connected. Denote $\lambda_2$ its second eigenvalue. Then
  \begin{align*}
    \mathbb{E}[V(\mathbf{S}(t))] \leq e^{\frac{-\lambda_2t}{N}}\mathbb{E}[V(\mathbf{S}(0))].
  \end{align*}
\end{theorem}
\begin{proof}
  From the previous lemmas we observe that
  \begin{align*}
    \frac{d}{dt}\mathbb{E}[V(\mathbf{S}(t))] = -\frac{1}{2N^2}\mathbb{E}[\sum_{ij}a_{ij}|S_j-S_i|^2] \leq -\frac{\lambda_2}{N}\mathbb{E}[V(\mathbf{S}(0))]
  \end{align*}
  Applying Gronwall's lemma gives the result.
\end{proof}

\begin{remark}
  Thus we see that the expected value of the variance of the opinions decays at rate $\frac{\lambda_2}{N}$; the dependence on the number of agents $N$ is naturally contrasted with the deterministic case where the decay rate is only $\lambda_2$, in other words, the deterministic dynamics converge much faster than the stochastic dynamics.
\end{remark}

Recall that in the deterministic case studied earlier we found that in the case of a balanced graph that the consensus converged to is the initial average of opinions.  We can make a related observation concerning the stochastic model by exploiting the fact that the average, ${\bar S}$, is a martingale. The following corollary also gives a succinct proof of the convergence of the stochastic dynamics in this case.

\begin{corollary}
	Suppose $G$ is a balanced directed graph \eqref{eq:balanced} and strongly connected. Then with probability 1, the dynamics $S(t)$ converge to a consensus opinion $S^*$ at time $T<\infty$ and the consensus satisfies $\mathbb{E}[S^*] = \frac{1}{N}\sum_1^N S(0)$.
\end{corollary}
\begin{proof}
  Let ${\bar S}(t) = \frac{1}{N} \sum_1^N S_i(t)$ the average. Using eq. \eqref{eq:generator_SDE}, we deduce:
  \begin{displaymath}
    \text{d} \mathbb{E}[{\bar S}(t)] = \sum_{i,j,j\neq i} a_{ij}\mathbb{E}[\Phi_{ij}({\bar S}(t))-{\bar S}(t)] = \sum_{i,j,j\neq i} \frac{a_{ij}}{N}\mathbb{E}[S_j(t) - S_i(t)] = 0
  \end{displaymath}
  since the graph is balanced (see remark \ref{rem:balanced}). Therefore ${\bar S}(t)$ is a martingale. Moreover, ${\bar S}(t)$ is bounded with at most $N^2$ states (finite) and therefore there exists a stopping time $T$ that is finite with probability $1$: ${\bar S}(t) \to {\bar S}^*$ almost surely.  Then, by the Optional Stopping Theorem,
  \begin{align*}
    \mathbb{E}[S^*] = \mathbb{E}[{\bar S}(T)] = \mathbb{E}[{\bar S}(0)] = {\bar S}(0)
  \end{align*}
  Since the graph $G$ is also strongly connected, we know by Theorem \ref{thm:cv_consensus_SCM} that $S^*$ must be a consensus.

\end{proof}

\section{Numerical simulations}

We will now present some numerical experiments to further analyze more subtle features of both of the consensus models studied.  We first examine explicitly how the \textit{algebraic connectivity} of the interaction network affects the speed of convergence in the case of a symmetric $L$ - recall that in both the deterministic and stochastic case that we were able to bound the speed of convergence in terms of the algebraic connectivity.  Then, we focus on the case where there are multiple isolated blocks as we have shown in this case that a consensus does \textit{not} form unconditionally.  We are interested in how the opinions of isolated blocks affect the distribution of opinions in absorbing components.  To this end we will examine the case of a ``battle'' between two isolated blocks with opinions at the extreme ends of the possible spectrum of opinions.  Each of these isolated blocks will have the same number of connections to a central absorbing component whose internal connections are all symmetric.

\subsection{Algebraic connectivity and speed of convergence}

  Earlier in the discussion of the deterministic model we showed that we could characterize the convergence rate of the model in terms of $\lambda_{2}$, the eigenvalue of $L$ with the second smallest real part.  In the case that $L$ was symmetric (an undirected network) we could say even more - that the rate of convergence of the dynamics was 
  at least of of order $\mathcal{O}(\mathrm{e}^{-\lambda_{2} t})$. The second eigenvalue $\lambda_{2}$ is known as the \textbf{algebraic connectivity} of the graph $G$, it measures in a sense {\it how well the graph is connected}.  
  In this interpretation, since $\lambda_{2}$ bounds the speed of convergence, the dynamics converge faster if the graph is {\it better connected}.  We illustrate with several examples. We give two examples of graphs with the same number of vertices ($|V|=20$) but different adjacency matrices. In Fig.~\ref{fig:graph1}--left, we present a graph where each vertex is connected with its $4$ neighbors. In Fig.~\ref{fig:graph1}-right, the graph is composed of two sub-graphs connected by only one link resulting in a {\it weak} connectivity. Indeed, computing the algebraic connectivity $\lambda_{2}$, we observe that the graph on the left has a higher value even though it has fewer edges.

To illustrate how this affects the speed of convergence, we perform two simulations of the dynamics starting from the same initial configuration but using the two graphs presented in figure \ref{fig:graph1}.  The results are given in Fig.~\ref{fig:opinion_over_time}. We observe that the right figure converges quickly to two ``clusters'' but then these clusters are very slow to move together. This is due to the {\it weak} connectivity of the graph. In Fig.~\ref{fig:variance_over_time}, we estimate the variance (denoted ${\bf V}(t)$ in the stochastic dynamics) of the opinions over time :
\begin{equation}
  \label{eq:var}
  {\bf v}(t) = \frac{1}{2N^2} \sum_{i,j} |s_j(t)-s_i(t)|^2.
\end{equation}
The variances are decaying to zero (indicating that a consensus is forming). As expected, the variance is decaying faster for the graph with the larger $\lambda_{2}$ (left-figure).  We also include the evolution of the stochastic dynamics on the graph in Fig.~\ref{fig:graph1}--right,  from the same initial condition, in Fig.~\ref{fig:variance_over_time}--left.  Notice that in Fig.~\ref{fig:variance_over_time}--right that the variance of the stochastic dynamics is the slowest to decay - this is an illustration of how the stochastic dynamics slow down the speed of convergence.

\begin{figure}[p]
  \centering
  \includegraphics[scale=1.2]{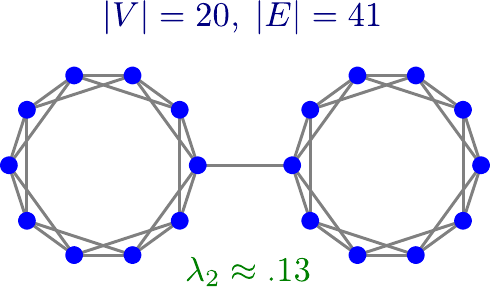} \qquad
  \includegraphics[scale=1.2]{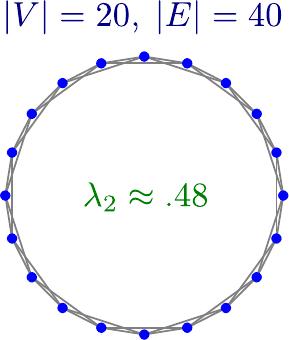}
  \caption{Examples of values of $\lambda_{2}$ for two graphs. Notice that even though the right graph has more edges, its algebraic connectivity $\lambda_2$ is lower.}
  \label{fig:graph1}
\end{figure}

\begin{figure}[p]
  \centering
  \includegraphics[width=.47\textwidth]{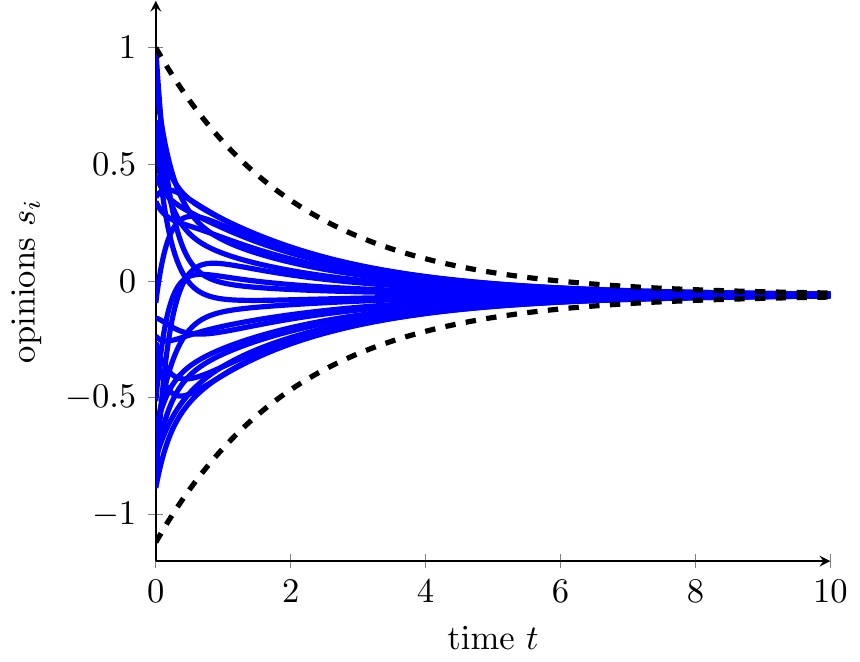}
  \includegraphics[width=.47\textwidth]{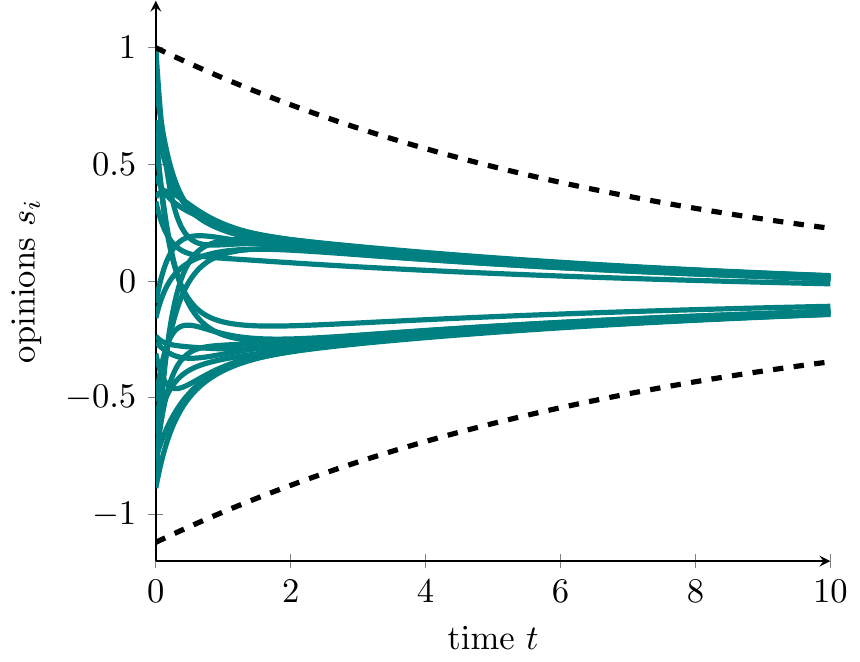}
  \caption{Evolution of opinions $s_i(t)$ over time for the two graphs of Fig.~\ref{fig:graph1}. The dash line indicates a decay with rate $\lambda_{2}$ (worst-case scenario).}
  \label{fig:opinion_over_time}
\end{figure}

\begin{figure}[p]
  \centering
  \includegraphics[width=.47\textwidth]{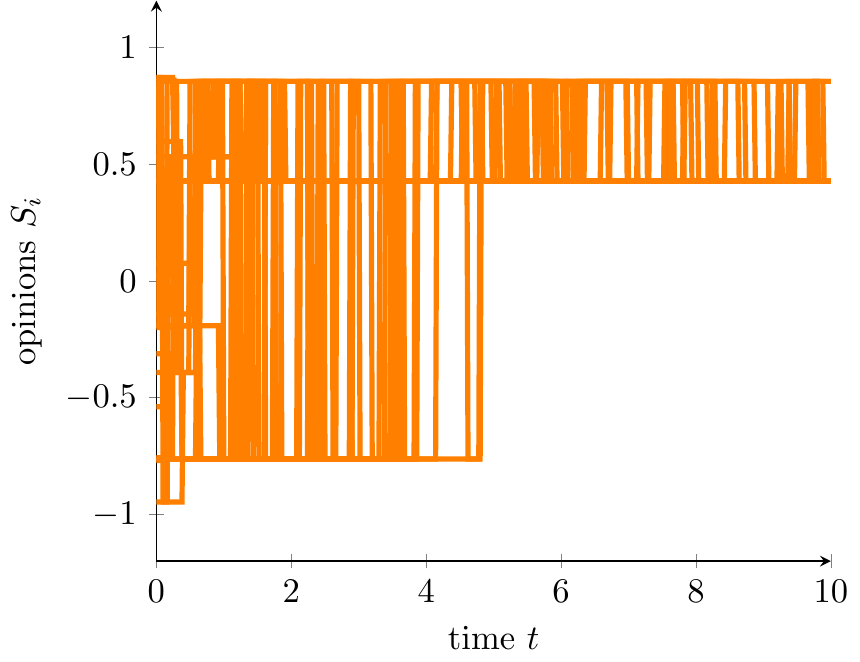}
  \includegraphics[width=.47\textwidth]{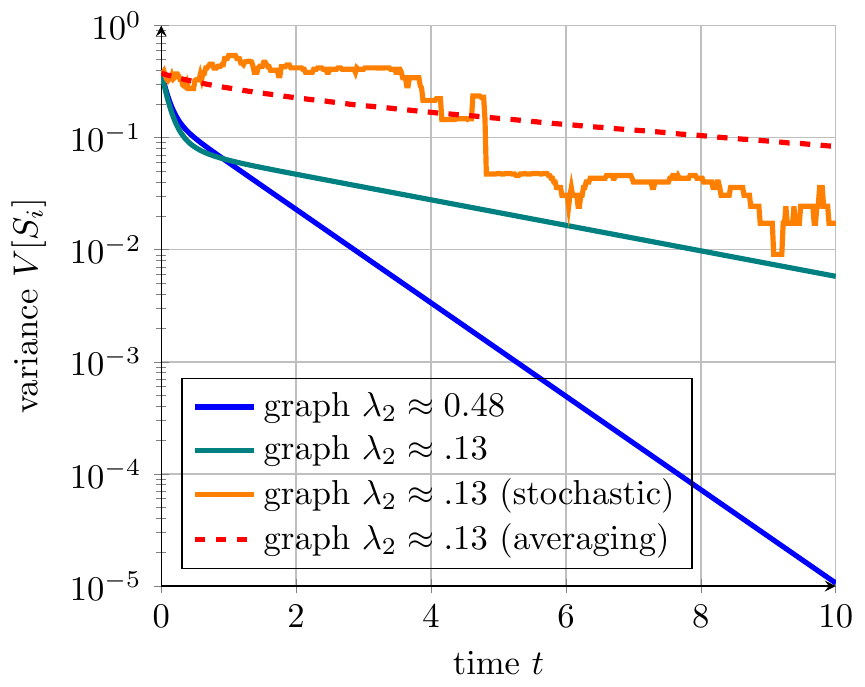}
  \caption{{\bf Left:} Evolution of opinions $S_i(t)$ for the stochastic dynamics on the graph with $\lambda_2\approx .13$. After $t=5$, the opinions $S_i(t)$ are jumping back and force between two values. They will eventually converge to a single value due to theorem \ref{thm:cv_consensus_SCM}. {\bf Right:} Evolution of the variance of opinions $V[S_i]$ for the three illustrations (i.e. Fig.~\ref{fig:opinion_over_time} and Fig.~\ref{fig:variance_over_time}-left). The decay is faster when $\lambda_2$ is larger. Moreover, we estimate the average decay of the variance for the stochastic dynamics $\mathbb{E}[V]$ (using 10,000 realization) plotted in red.}
  \label{fig:variance_over_time}
\end{figure}

\subsection{Outside of consensus - multiple isolated blocks}

Here we examine numerically the effects of more then one isolated block on the long-term behavior of both models.  We set up the initial opinions so that each isolated block is at each extreme of the possible spectrum of opinions and that opinions in the central component are randomly distributed uniformly among the possible opinions.

\begin{figure}[ht!]
    \centering
    \includegraphics[scale = .7]{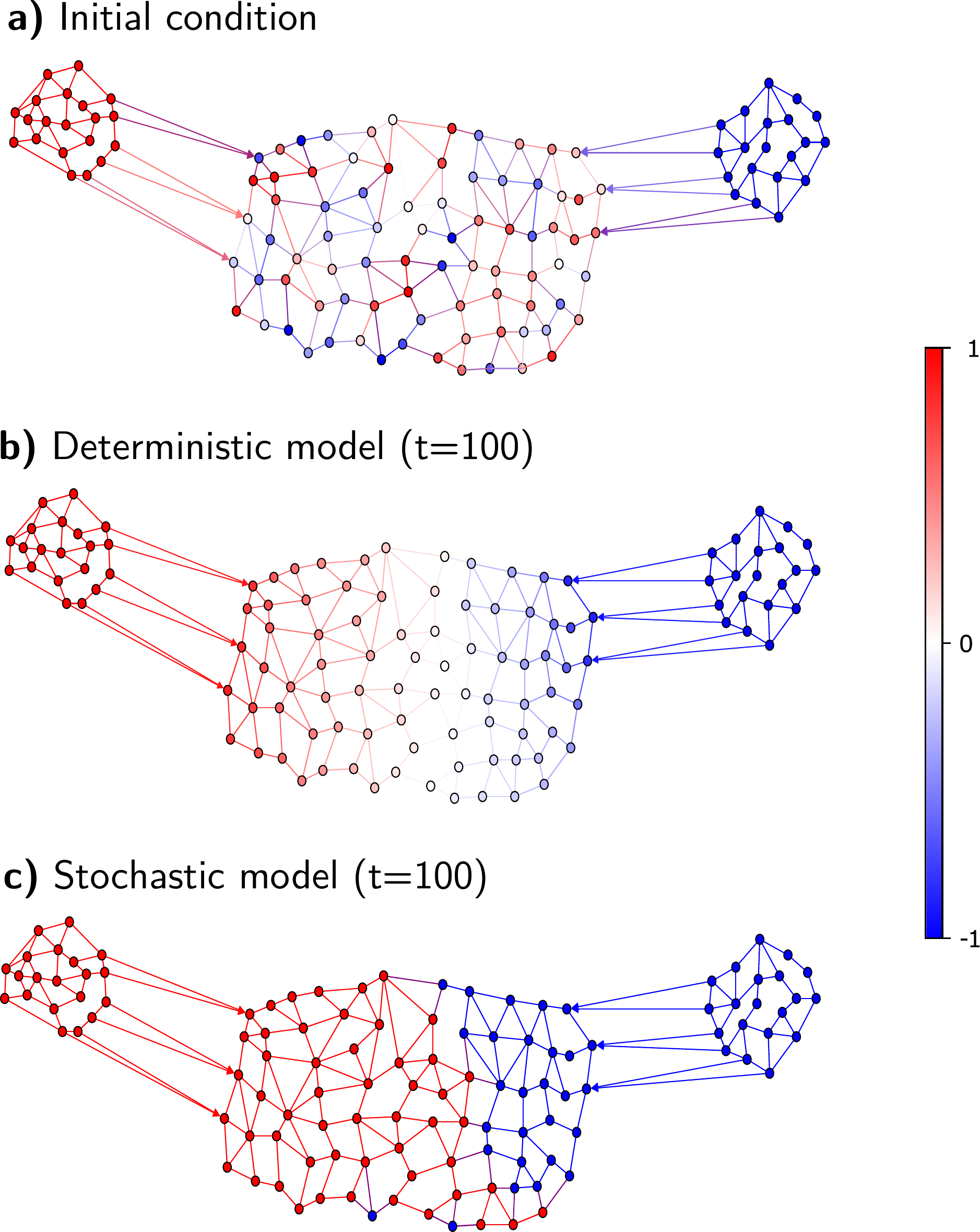}
    \caption{a) Initial state of the network in the "battle" scenario. b) The final state of the network in the "battle" scenario for the deterministic model.  Nodes close to isolated blocks exhibit an opinion close to that of the isolated block. c) The state of the network at $t=100$ for the stochastic model.}
    \label{fig:big_fight}
\end{figure}

\begin{figure}[ht!]
    \centering
    \includegraphics[scale = .7]{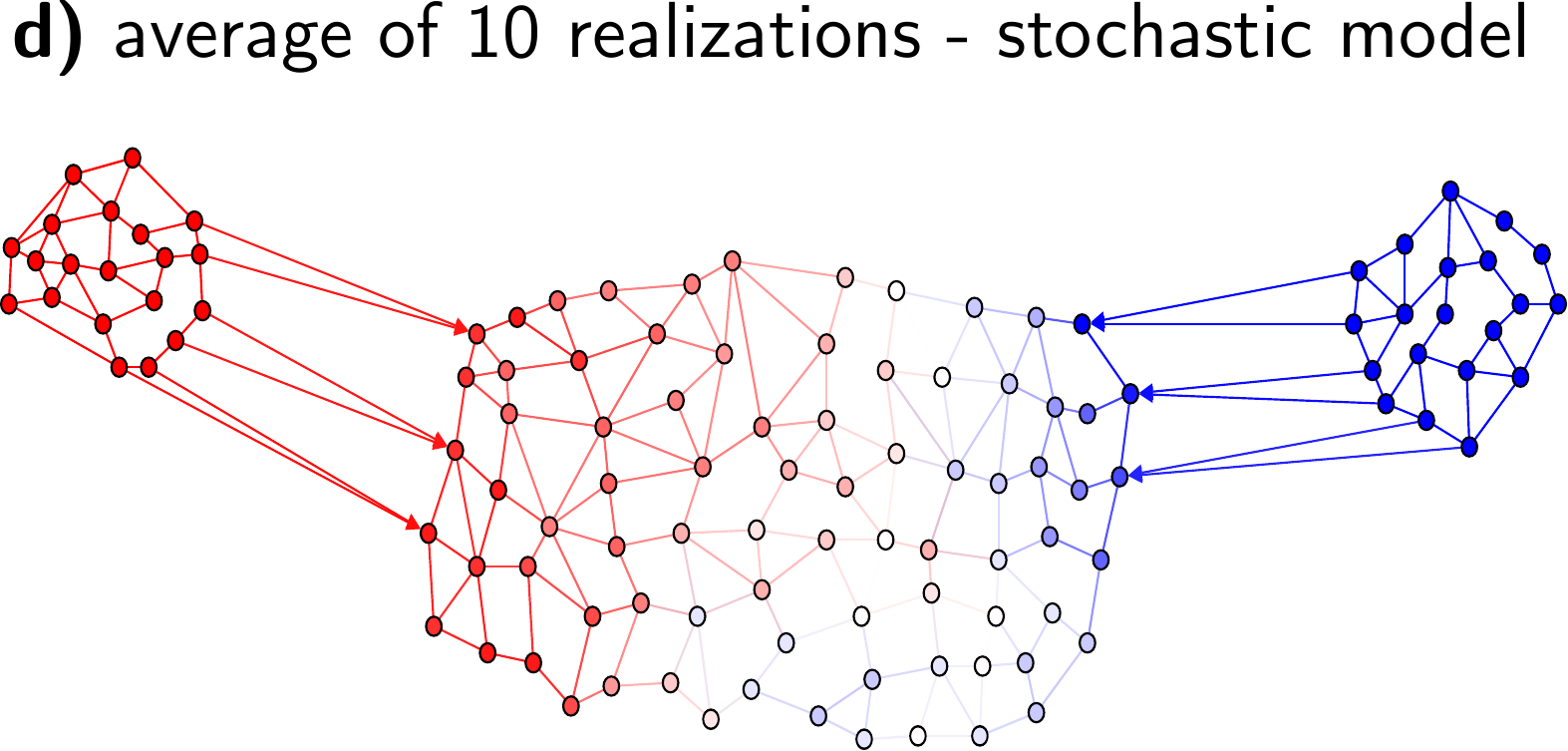}
    \caption{d) The average of $10$ realizations at $t=100$.  Notice that analogous to the deterministic case, that nodes close to an isolated block have an opinion close to the opinion of the isolated block.}
    \label{fig:big_fight_d}
\end{figure}

After running a simulation of the dynamics we find that the middle component now exhibits a ``gradient'' of the possible opinions.  Nodes close to the blue isolated block in the sense of graph distance exhibit an opinion close to that of the isolated block, likewise for nodes near the red isolated block.  Nodes in the middle of the component exhibit an opinion close to neutral.  This shows that in the case of multiple isolated blocks that the graph distance of a node in an absorbing component from isolated blocks plays a role in determining the final opinion of that node.


We now examine the analogous situation for the stochastic model using the same initial condition. We have already seen in the theory of the model that the agents in the central component cannot converge  to a limiting opinion distribution as there are multiple isolated blocks in the network.  However, to get an idea of how the network evolves in the stochastic case we show the state at $t=100$.

Notice that similar to the deterministic case that it appears that nodes closer to an isolated block in the graph sense are likely to have the same opinion as that isolated block.  However, this could very well be an artifact of the realization shown.  To see that this phenomenon does indeed occur in the stochastic case we recall that we have shown that while the dynamics do not converge in this case, they are the same as the deterministic dynamics \textit{in expected value}.

 If we take the mean of 10 realizations at $t=100$ we recover a gradient of opinions in the central component similar to what we found in the deterministic case.  This indicates that indeed if a node is closer to an isolated block in the graph sense then it is more likely to have the opinion of that isolated block and serves to illustrate that the dynamics of the stochastic model coincide with the deterministic model in expected value.


\section{Conclusion and future work}

We have discussed consensus models in both a deterministic and stochastic setting.  In both cases we have completely characterized the convergence of the model. In the deterministic case we find that the model always converges although not necessarily to a consensus; this contrasts with the stochastic case where we find that the model only converges almost surely if and only if the network has only one isolated block.  We have also given conditions that guarantee unconditional convergence to a consensus - isolated blocks play a critical and identical role in both cases.
We find that for both models unconditional convergence to a consensus is equivalent to the network on which the model is defined having only one isolated block.  This similarity is explained by the fact that the deterministic dynamics can be recovered from the stochastic dynamics in expectation.  Finally, we show in the deterministic case that the rate of convergence depends only on the algebraic connectivity of the network - the stronger the connectivity the faster the model converges.  We find a similar result in the stochastic model in the case of symmetric networks however in this case we see that the rate of convergence is much slower as it also depends on the number of agents.

There are many ways in which this study may be extended.  We hope to make our assumptions about agent-agent interactions more subtle through adding non-linearities to the model.  First, in this paper we assume that the network is static; the connections between individuals do not change with respect to time.  To further generalize the consensus model we could study its behavior on graphs in which the connections between individuals are evolving with time and depend on the current opinions of individuals. A consensus model assumes that if two individuals interact for a sufficient amount of time they will reach a consensus,  however everyday experience tells us that this is not always (or even often) the case.  These considerations motivate the study of the following class of models:
\begin{equation}
  \label{eq:bidon}
  s_{i}'(t) = \sum_{j \neq i}a_{ij}(t,s_{i},s_{j})\phi(s_{j}(t),s_{i}(t)).
\end{equation}
$\phi$ is known as an \textit{influence function} and dictates how the opinion of individual $i$ changes based on its current value and the opinion of those who that individual is connected with.
We also note that the behavior exhibited by the consensus model is reminiscent to that of a transport-diffusion model in the case of ``grid-network'' (for example those studied in Section 4.2).  We hope to characterize this behavior through passing to a continuous description of the graph to obtain in a certain limit a partial differential equation describing the dynamics.  One possible avenue to this description would be through a kinetic description of the model combined with a coarsening approach \cite{dornic_critical_2001, herty_large_2014}. However, as it has been observed \cite{carlen_kinetic_2013}, propagation of chaos is not always applicable to these dynamics which could hinder this approach.

%

Empirical studies of models of collective behavior outside of physics have traditionally been difficult due to lack of data and complexity of the interacting agents.  However, progress has been made in biological models \cite{ballerini_empirical_2008, buhl_disorder_2006} and we believe that the advent of social media and the simplicity of the consensus model make an empirical study of consensus dynamics feasible. Ideally the graph of interaction will be constructed from a source of real-world networks; for example, the meta data associated with social media \cite{estrada_communicability_2015, kempe_maximizing_2003}.  A possible goal of such a study would be to confirm the phenomenon found in this study; connectivity in the graph is a main factor in determining how effectively influence spreads through a network.


\appendix

\section{Appendix}

\subsection{Convergence of linear systems}
\label{sec:cv_linear}
\begin{lem}
  Given a linear system defined by:
  \begin{displaymath}
    \vect{x'}=A\vect{x} \quad\vect{x}(0)=\vect{x}_{0}
  \end{displaymath}
  Assume $A$ has $d$ distinct eigenvalues and a zero eigenvalue of multiplicity $m$ with $m$ linearly independent associated eigenvectors. If for all $\lambda_{i}\in\text{Spec}(A)$ with $i>m$ we have that $Re(\lambda_{i})<0$ then
  \begin{displaymath}
    \lim_{t\rightarrow\infty}\vect{x}(t)=\vect{u}
  \end{displaymath}
  where $\vect{u}$ is in the center subspace of $A$, $E^{c}$.  
\end{lem}

\begin{proof}
 For ease of notation we will write
  \begin{align*}
          \text{Spec}(A) = \set{\lambda_{i},..,\lambda_{d}}
 \end{align*}
 where $m_{i}$ is the algebraic multiplicity of $\lambda_{i}$. We will also write that $\lambda_{1} = 0$. We know that given $\lambda_{i}$ that we may find $m_{i}$ linearly independent generalized eigenvectors of $A$, $\set{{\bf v}_{\lambda_{i}}^{1},...,{\bf v}_{\lambda_{i}}^{m_{i}}}$.  The generalized eigenspace corresponding to $\lambda_{i}$ is given by $E_{\lambda_{i}} = \text{span}\set{{\bf v}_{\lambda_{i}}^{1},...,{\bf v}_{\lambda_{i}}^{m_{i}}}$ and by the generalized eigenspace decomposition theorem
we can find a basis of $\Rn$ consisting of generalized eigenvectors of $A$.  Therefore, we may write:

\begin{equation}\label{directsum}
     \Rn = \bigoplus_{i=1}^{d}E_{\lambda_{i}}.
\end{equation}
 We know by the fundamental theorem of linear systems that the solution to the system is given by:
 \begin{align*}
     \vect{x}(t) = e^{At}\vect{x}_{0}.
 \end{align*}
 By (\ref{directsum}) we may write:
 \begin{align*}
 \vect{x}_{0} = {\bf e}_{1}+...+{\bf e}_{d}
 \end{align*}
 where for each $1\leq i \leq d$ we have ${\bf e}_{i}\in E_{\lambda_{i}}$.  Notice that since we are given that there are $m_{1}$ distinct eigenvectors associated with $\lambda_{1} = 0$ that for each $w\in E_{\lambda_{1}}$ we have that $A{\bf w}=0$. Consequently we have that $A_{|E_{\lambda_{1}}}=0$ which implies that $A_{|E_{\lambda_{1}}}=0$ for each $t\in\R$.
Taking the matrix exponential of both sides yields
 \begin{align*}
     e^{tA_{|E_{\lambda_{1}}}}&=\text{Id}\quad\text{for each $t\in\R$.}\\
 \end{align*}
 Therefore we must have that
 \begin{align*}
     e^{At}\vect{x}_{0} &= e^{At}{\bf e}_{1}+...+e^{At}{\bf e}_{d} \\                        &=e^{tA_{|E_{\lambda_{1}}}}{\bf e}_{1}+...+e^{tA_{|E_{\lambda_{d}}}}{\bf e}_{d}\\                        &={\bf e}_{1}+e^{tA_{|E_{\lambda_{2}}}}{\bf e}_{2}+...+e^{tA_{|E_{\lambda_{d}}}}{\bf e}_{d}.
 \end{align*}
 We now claim that
 \begin{align*}
    \lim_{t\rightarrow\infty}e^{tA_{|E_{\lambda_{i}}}}{\bf e}_{i}=0.
\end{align*}
  for each $1< i\leq d$. Since we are writing $\vect{x_{0}}$ with respect to the basis of generalized eigenvectors of A we have that
 \begin{align*}
     A_{|E_{\lambda_{i}}} = \lambda_{i}Id + N\\
 \end{align*}
 where N is nilpotent. Consequently
 \begin{align*}
     e^{tA_{|E_{\lambda_{i}}}}e_{i} &= e^{t(\lambda_{i}\text{Id} +N)}e_{i} \\
     &=e^{t\text{Re}(\lambda_{i})}e^{t\text{Im}(\lambda_{i})i}(Id+Nt+...+\frac{1}{k!}t^{k}A^{k})e_{i}\\
     &\leq C_{1}e^{t\text{Re}(\lambda_{i})}(Id+Nt+...+\frac{1}{k!}t^{k}A^{k})e_{i}
 \end{align*}
 with $C_1>0$.  Therefore, since $\text{Re}(\lambda_{i})<0$ and every coordinate of $(Id+Nt+...+\frac{1}{k!}t^{k}A^{k}){\bf e}_{i}$ is polynomial in $t$ we must have that there exist $C>0$ and $0<\epsilon<-\text{Re}(\lambda_{i})$ such that
 \begin{equation} \label{eq:convergence_speed}
     \norm{e^{tA_{|E_{\lambda_{i}}}}{\bf e}_{i}} \leq Ce^{(\text{Re}(\lambda_{i})+\epsilon)t}\rightarrow 0 \quad\text{as}\quad t\rightarrow\infty.
 \end{equation}
 So we have that
 \begin{align*}
     \lim_{t\rightarrow\infty}e^{tA_{|E_{\lambda_{i}}}}{\bf e}_{i}=0\quad 1<i\leq d,
 \end{align*}
which implies
\begin{align*}
    \lim_{t\rightarrow\infty}e^{At}\vect{x_{0}} = \lim_{t\rightarrow\infty}&={\bf e}_{1}+e^{tA_{|E_{\lambda_{2}}}}{\bf e}_{2}+...+e^{tA_{|E_{\lambda_{d}}}}{\bf e}_{d} = {\bf e}_{1}.
\end{align*}
As ${\bf e}_{1}\in E^{c}$ by definition we have that $\lim_{t\rightarrow\infty}\vect{x}\in E^{c}$ as desired.
\end{proof}

\subsection{Decomposition into strongly connected components}
\label{sec:frobenius}

\begin{lem}
  If L is the Laplacian of a directed graph $G=(V,E)$ then by relabeling vertices L can be represented:
\end{lem}
\begin{align*}
  L=
  \begin{bmatrix}
    B_{1}                                    \\
    &  B_{2}             &   & \text{\huge0}\\
    &               & B_{3}                \\
    & \text{\huge*} &   & \ddots            \\
    &               &   &   & B_{k}
  \end{bmatrix}
\end{align*}
\begin{proof}
  We may partition the vertex set of G, $V$, into its strongly connected components $\set{U_{1},\dots,U_{k}}$, so that:
  \begin{align*}
    V=\uplus_{i}^{k}U_{i}.
  \end{align*}
  If we consider the set $\set{U_{1},\dots,U_{k}}$ as the vertex set of a new graph $G^{*}$ with edge set $E^{*}$ given by:
  \begin{align} \label{big_edges}
    (U_{m},U_{n})\in E^{*} \text{ if there exists } u\in U_{m} \text{ and } v\in U_{n}\quad\text{with}\quad (u,v)\in E.
  \end{align}
  Then the graph $G^{*}$ is a directed acyclic graph as $G$ has been partitioned into strongly connected components; if a cycle existed all vertices included in the cycle would represent the same strongly connected component of $G$ by \eqref{big_edges} contradicting the partition of $V$ into strongly connected components. Therefore there exists a topological ordering $\leq^{*}$ on the vertex set of $G^{*}$.  This ordering is given by:
  \begin{align*}
    U_{m}\leq^{*}U_{n}\quad\text{if}\quad(U_{m},U_{n})\in E^{*}.
  \end{align*}
  We only need to label the vertices of $V$ in such a way that they respect the topological ordering. Then, this labeling of $V$ produces $L$ in the desired form.
\end{proof}


\end{document}